\documentclass[12pt]{article}
\usepackage{amsmath, latexsym, amsfonts, amssymb, amsthm, amscd}
\usepackage[left=3cm, right=2.5cm, top=2.5cm, bottom=2.5cm]{geometry}
\usepackage{pdfsync}
\usepackage{color}

\newtheorem{corollary}{Corollary}
\newtheorem{proposition}{Proposition}
\newtheorem{lemma}{Lemma}

\newtheorem{theorem}{Theorem}

\newcommand{\p}{\mathbb{P}}
\newcommand{\e}{\mathbb{E}}
\newcommand{\ud}{\mathrm{d}}

\newcommand{\R}{\mathbb{R}}

\newcommand{\N}{\mathbb{N}}

\newcommand{\Ind}[1]{\mathbf{1}_{\{#1\}}}
\newcommand{\Expo}[1]{\exp\left\{#1\right\}}
\newcommand{\Prob}[1]{\mathbb{P}\left(#1\right)}
\newcommand{\Exp}[1]{\mathbb{E}\left[#1\right]}
\newcommand{\Expx}[2]{\mathbb{E}_{#1}\left[ #2 \right]}
\newcommand{\Expconx}[3]{\mathbb{E}_{#1}\left[#2\big| #3 \right]}

\begin{document}

\title{Branching processes in a L\'evy random environment.}
\author{ S. Palau\footnote{ {\sc Centro de Investigaci\'on en Matem\'aticas A.C. Calle Jalisco s/n. 36240 Guanajuato, M\'exico.} E-mail: sandra.palau@cimat.mx. Corresponding author}\,\, and J.C. Pardo\footnote{ {\sc Centro de Investigaci\'on en Matem\'aticas A.C. Calle Jalisco s/n. 36240 Guanajuato, M\'exico.} E-mail: jcpardo@cimat.mx}\\}
\maketitle
\vspace{0.2in}

\begin{abstract} 
In this paper, we introduce branching processes in a L\'evy random environment. In order to define this class of processes, we study a  particular class of non-negative stochastic differential equations driven by   Brownian motions and  Poisson random measures which are mutually independent. The existence and uniqueness of strong solutions are established under some general conditions that allows us to consider  the case when the strong solution explodes at a finite time.  We use the latter result to construct  continuous state branching processes with immigration and competition in a L\'evy random environment as a strong solution of a stochastic differential equation. We also study some properties of such processes  that extends recent results obtained by Bansaye et al. in (Electron. J. Probab. 18, no. 106, 1-31, (2013)), Palau and Pardo  in (arXiv:1506.09197 (2015)) and Evans et al. in (J. Math. Biol., 71, 325-359, (2015)).

\bigskip


\bigskip

\end{abstract}

\vspace{0.5cm}

\section{Introduction}

In many biological systems, when the population size is  large enough, many birth and death events occur. Therefore, the dynamics of the population become difficult to describe. Under this scenario,  continuous state models are good approximations of these systems and sometimes they can be simpler and computationally more tractable. Moreover, the qualitative behaviour of the approximate models may be easier to understand. 

The simplest branching  model in continuous time and space  is perhaps the so called continuous state branching process (or CB-process for short). This model arises as the limit of Galton-Watson processes; where individuals behave independently one from each other and each individual gives birth to a random number of offspring, with the same offspring distribution (see for instance Grimvall \cite{gri}).  More precisely, a CB-process is a $[0,\infty]$-valued strong Markov process $Y=(Y_t,t\geq 0)$ with c\'adl\'ag paths such that satisfies the branching property: for all $\theta\geq 0$ and $x,y\geq 0$,
$$\Expx{x+y}{e^{-\theta Y_t}}=\Expx{x}{e^{-\theta Y_t}}\Expx{y}{e^{-\theta Y_t}}.$$
Moreover, its law  is completely characterized by the latter identity, i.e.
\begin{equation*}
  \Expx{x}{e^{-\lambda Y_t}}=e^{-xu_t(\lambda)},\qquad  t\geq 0,
\end{equation*}
where $u$ is a differentiable  function in $t$ satisfying
\begin{equation}
\frac{\partial u_{t}(\lambda)}{\partial t}=-\psi(u_{t}(\lambda)), \qquad u_{0}(\lambda)=\lambda,
\label{DEut}
\end{equation}
and $\psi$, the branching mechanism of $Y$, satisfies 
\begin{equation*}
\psi(\lambda)=-q-a\lambda+\gamma^2
\lambda^2+\int_{(0,\infty)}\big(e^{-\lambda x}-1+\lambda x{\mathbf 1}_{\{x<1\}}\big)\mu(\ud x),\qquad \lambda \geq 0,
\end{equation*}
where $a\in \mathbb{R}$, $q,\gamma\geq 0$ and $\mu$ is a
 measure concentrated on $(0,\infty)$ such that
$\int_{(0,\infty)}\big(1\land x^2\big)\mu(\ud x)$ is finite. A process in this class can also be defined as the unique non-negative strong solution of the following stochastic differential equation (SDE for short) 
\begin{equation*}
\begin{split}
Y_t=&Y_0+a\int_0^t Y_s \ud s+\int_0^t \sqrt{2\gamma^2 Y_s}\ud B_s \\
&\,\,\,+\int_0^t\int_{(0,1)}\int_0^{Y_{s-}}z\widetilde{N}(\ud s,\ud z,\ud u)+\int_0^t\int_{[1,\infty]}\int_0^{Y_{s-}}zN(\ud s,\ud z,\ud u),
\end{split}
\end{equation*}
where $B=(B_t,t\geq 0)$ is a standard Brownian motion, $N(\ud s,\ud z,\ud u)$ is a Poisson random measure independent of $B$, with intensity $\ud s\Lambda(\ud z)\ud u$ where 
$\Lambda(\ud z)=\mu(\ud z)\mathbf{1}_{(0,\infty)}(z)+q\delta_\infty(\ud z)$ and $\widetilde{N}$ is the compensated measure of $N$, see for instance \cite{FuLi}. 

Recently there has been some interest in extending this model,  in the sense that one would like to  include immigration, competition or dependence on the environment. This interest comes from the fact that these new models arise as  limits of discrete population models where there are interactions between individuals or where the offspring distribution depends on the environment (see for instance Lambert \cite{Lambert1}, Kawasu and Watanabe \cite{KW}, Bansaye and Simatos \cite{basim}). 

Recall that a CB-process with immigration (or CBI-process) is a strong Markov process taking values in $[0,\infty]$, where 0 is no longer an absorbing state. It is characterized by a branching mechanism $\psi$
 and an immigration mechanism, $$\phi(u)=\mathtt{d}u+\int_0^{\infty} (1-e^{-u t})\nu(\ud t),\qquad u\geq 0,$$
  where $\mathtt{d}\geq 0$ and $\nu$ is a measure supported in $(0,\infty)$ satisfying $ \int_0^{\infty}(1\wedge x)\nu(\ud x)<\infty.$
It is well-known that if $(Y_t, t\ge 0)$ is a process in this class, then its semi-group is characterized by
$$\Expx{x}{e^{-\lambda Y_t}}=\Expo{-xu_t(\lambda)-\int_0^t\phi( {u_s(\lambda)})\ud s}, \qquad \mbox{for} \quad \lambda\geq 0,$$
where $u_t$ solves (\ref{DEut}).

According to Fu and Li \cite{FuLi}, under the condition that $\int_{(0,\infty)}( x\land x^2)\mu({\rm d} x)$ is finite,  a CBI-process can be defined as the unique non-negative strong solution of the stochastic differential equation
\begin{align*}
Y_t=Y_0+&\int_0^t(\mathtt{d}+a Y_s) \ud s+\int_0^t \sqrt{2\gamma^2 Y_s}\ud B_s\\
&\hspace{2cm}+\int_0^t\int_{(0,\infty)}\int_0^{Y_{s-}}z\widetilde{N}(\ud s,\ud z,\ud  u)+\int_0^t\int_{(0,\infty)}z M^{(im)}(\ud s,\ud z),
\end{align*}
where  $M^{(im)}(\ud s,\ud z)$ is a Poisson random measures with intensity $\ud s\nu(\ud z)$ independent of  $B$ and  $N$.

CB processes  with competition  were first studied  by Lambert \cite{Lambert1}, under the name of logistic branching processes,   and more recently studied by Ma \cite{Ma} and Beresticky et al. \cite{BFF}. Under the assumptions that  \[
q=0\qquad  \textrm{and} \qquad \int_{(0,\infty)}\big(x\land x^2\big)\mu(\ud x)<\infty,
\]
 the CB-process with competition is defined as the unique strong solution of the following  SDE 
\begin{align*}
Y_t=&Y_0+a\int_0^t Y_s \ud s-\int_0^t \beta(Y_s)\ud s+\int_0^t \sqrt{2\gamma^2 Y_s}\ud B_s +\int_0^t\int_{(0,\infty)}\int_0^{Y_{s-}}z\widetilde{N}(\ud s,\ud z,\ud u),
\end{align*}
where $\beta$ is a continuous non-decreasing function on $[0,\infty)$ with $\beta(0)=0$, which is  called the competition mechanism. The interpretation of the function $\beta$ is the following: in a given population of size $z$, an additional individual would be killed a rate $\beta(z)$.

Branching processes in random environment (BPREs) were first introduced and studied in Smith and Wilkinson \cite{smith} and have attracted considerable interest in the last
decade (e.g. \cite{afa12,bo10,liu,vatuti} and the reference therein). BPREs are interesting since they are more realistic models compared with classical branching processes and, from the mathematical point of view, they have new properties such as phase transitions in the subcritical regime. 
Scaling limits in the finite variance case were conjectured by Keiding \cite{kei} who introduced Feller diffusions in random environment. This conjecture was proved by Kurtz \cite{Kurtz} and by Bansaye and Simatos \cite{basim} in  more general cases. 

There are new studies about the continuous state space setting, in all of them, the CB-process in random environment is defined as a strong solution of a particular  {SDE}. B\"oinghoff and Hutzenthaler \cite{CbMh} studied the case when the process possesses continuous paths. This process is the strong solution of the SDE
\begin{equation}\label{sdeenvironment}
 Z_t=Z_0+a\int_0^t Z_s ds+\int_0^t \sqrt{2\gamma^2 Z_s}dB_s +\int_0^t  Z_{s-}dS_s,
 \end{equation}
where the process $S=(S_t,t\ge 0)$ is a Brownian motion with drift which is independent of $B$. Bansaye and Tran \cite{BT} studied a cell dividing model which are infected by parasites. Informally, the quantity of parasites in a cell evolves as a Feller diffusion. The cells divide in continuous time at  rate $r(x)$, which may depend on the quantity of parasites $x$ that they contain. When a cell divides, a random fraction $\theta$ of parasites goes in the first daughter cell and the rest in the second one. In each division, they only keep one cell and consider the quantity of parasites inside. Assuming that the rate $r$ is constant  and $\theta$ is a r.v. in $(0,1)$ with distribution $F$, the model follows a Feller diffusion with multiplicative jumps of independent sizes distributed as $F$ and which occurs at rate $r$. In particular, the model can be described as in (\ref{sdeenvironment}) with $S$ satisfying
\begin{align*}
S_t=-r\int_0^t\int_{(0,1)}(1-\theta)M(\ud s,\ud \theta)
\end{align*}
where  $M$ is a Poisson random measure with intensity $\ud sF(\ud \theta).$ Inspired in this model, Bansaye et al. \cite{Bapa} studied more general CB-processes in  random environment which  are driven by  L\'evy processes whose paths are of bounded variation and under the assumption that $\int_{(1,\infty)}x\mu(\ud x)<\infty$.  They were called   CB-processes with catastrophes motivated by the fact that the presence of a  negative jump in the random environment represents that a proportion of a population, following  the dynamics of the CB-process,  is killed. The process is  {defined} as the unique non negative strong solution of the  {following} SDE
\begin{align*}
Z_t=&Z_0+a\int_0^tZ_s \ud s+\int_0^t \sqrt{2\gamma^2 Z_s}\ud B_s+\int_0^t\int_{(0,\infty)}\int_0^{Z_{s-}}z\widetilde{N}(\ud s,\ud z,\ud u) +\int_0^t  Z_{s-}\ud S_s,
\end{align*}
where
\begin{align*}
S_t=\int_0^t\int_{(0,\infty)}(m-1)M(\ud s,\ud m),
\end{align*}
$M$ is a Poisson random measure independent of $N$ and $B$, with intensity $\ud s\nu(\ud m)$, where $\nu$ is a measure concentrated on $(0,\infty)$ such that
$$0<\int_{(0,\infty)}(1\wedge |m-1|)\nu(\ud m)<\infty.$$
Palau and Pardo \cite{PP}  {considered} a general CB-process with immigration in a Brownian random environment. 
In other words, the authors in \cite{PP} consider the following SDE,
\begin{align*}
Z_t=&Z_0 +\int^t_0\left(\mathtt{d} +aZ_s\right)\ud s +\int_0^t \sqrt{2\gamma^2 Y_s}\ud B_s+\int_0^t  Z_s\ud S_s\nonumber\\
&\hspace{2cm}+\int_0^t\int_{(0,\infty)}\int_0^{Z_{s-}}z\widetilde{N}(\ud s,\ud z,\ud u)+\int_0^t\int_{(0,\infty)}zM_1(\ud s,\ud z),
\end{align*}
where  $S$ is a Brownian motion with drift, $\mathtt{d}\geq 0$ and  $M_1(\ud s,\ud z)$ is a Poisson random measure with intensity $\ud s\nu(\ud z)$ satisfying
\[
\int_{(0,\infty)}(1\wedge x)\nu(dx)<\infty.
\]
 Moreover, the processes $S, B,N$ and $M_1$ are independent of each other.
In all these manuscripts, the authors proved the existence of such process and obtained the speed of extinction. As in the case of BPREs, there is a phase transition in the subcritical regime.

Evans et al. \cite{EHS} consider a population living in a spatially heterogeneous environment with $n$ different patches. These patches may represent distinct habitats, patches of the same habitat type or combinations thereof. The population in the $i$-th patch at time $t\geq 0$ is given by
$$X_t^{(i)}=X_0^{(i)}+\int_0^t X_s^{(i)}(\mu_i-k_iX_s^{(i)})\ud s+\int_0^t X_s^{(i)}\ud E_s^{(i)},$$
where $\mu_i$ is the rate of growth in the patch $i$, $k_i$ represents the  competition in the patch $i$, and $E_t^{(i)}=\sum_j \gamma_{ij}B_t^{(j)}$, with  $(B_t^{(1)},\cdots,B_t^{(n)})$  a standard Brownian motion on $\R^n$. They assume that the fraction of population in  {patch} $i$ is equal to $\alpha_i$ all the time. Therefore if $\alpha_1,\cdots,\alpha_n\geq 0$ are such that $\sum_{i\leq n}\alpha_i=1$, we have $X_t^{(i)}=\alpha_i X_t$ where $X_t=\sum_{i\leq n}X_t^{(i)}.$ In this case, the process $X$ satisfies the SDE
$$X_t=X_0+\int_0^t X_s\sum_{i\leq n}\alpha_i(\mu_i-k_i\alpha_iX_s)\ud s+\int_0^t X_s\sum_{i\leq n}\alpha_i\ud E_s^{(i)}.$$

In this paper, one of our aims is to construct a continuous state branching processes with immigration in a L\'evy random environment as a strong solution of a  {SDE}. In order to do so,  we study a particular class of non-negative SDE's driven by Brownian motions and Poisson random measures which are mutually independent. The existence and uniqueness of strong solutions are established under some general conditions that allows us to consider the case when the strong solution explodes at a finite time. This result is of particular interest on its own.

The remainder of the paper is structured as follows. In Section 2, we  study strong solutions of  {SDE's} which are driven by a finite number of Brownian motions and Poisson random measures which are mutually independent. Section 3 is devoted to the  construction of CBI-processes with competition in a L\'evy random environment which is an extension of the models introduced in Bansaye et al.  \cite{Bapa} and Palau and Pardo  \cite{PP}. In particular, we study the long term behaviour of CB-processes in a L\'evy random  environment. We finish our exposition by studying  a population model with competition in a L\'evy random environment which can be considered as an extension  of the model of Evans et al. \cite{EHS}. In particular, we study its long time behaviour and the Laplace transform of its first passage time below a level under the assumption that the environment has no negative jumps.

\section{Stochastic differential equations}
Stochastic differential equations  with jumps have been playing an ever more important role in various domains of applied probability theory such as financial mathematics or mathematical biology. Under Lipschitz conditions, the existence and uniqueness of strong solutions of SDE's with jumps can be established by arguments based on Gronwall's inequality and  results on continuous-type equation, see for instance the monograph of Ikeda and Watanabe \cite{IW}. In view of the results of Fu and Li \cite{FuLi} and Dawson and Li \cite{DL} weaker conditions would be sufficient for the existence and uniqueness of strong solutions for one-dimensional equations.  

Fu and Li \cite{FuLi} motivated by describing CBI processes via SDE's, studied general SDE's that describes  non-negative processes with jumps under   ge\-ne\-ral conditions. The authors in \cite{FuLi} (see also \cite{DL,LiPu}) provided criteria for the existence and uniqueness of strong solutions of those equations. The main idea of their criteria is to assume a monotoni\-ci\-ty condition on the kernel associated with the compensated noise so that the continuity conditions can be weakened.  Nonetheless, their criteria do not include the case where the branching mechanism of a CBI process has infinite mean and also the possibility of including a general random environment. This excludes some interesting models that can be of particular interest for applications.

Our goal in this section is to describe a general one-dimensional SDE that may relax this moment condition of Fu and Li \cite{FuLi} and also include some extra randomness that can help us to define branching processes in more general  random environment that those considered by Bansaye et al. \cite{Bapa} and Palau and Pardo \cite{PP}.

For $m, d,  l\ge 1$,  we define  the index sets $I=\{1,\dots, m\}$, $J=\{1,\dots, l\}$ and $K=\{1,\ldots, d\}$,  and take  $(U_i)_{i\in I}$  and $(V_j)_{j\in J}$ be separable topological spaces whose topologies can be defined by complete metrics. Suppose that $(\mu_i)_{i\in I}$ and $(\nu_j)_{j\in J}$ are $\sigma$-finite Borel measures such that each $\mu_i$ and $\nu_j$ are defined on  $U_i$ and  $V_j$, respectively.  We say  {that} the parameters $(b, (\sigma_k)_{k\in K},  (h_i)_{i\in I}, (g_j)_{j\in J})$ are \textit{admissible} if
\begin{itemize}
\item[i)] $b: \R_+\rightarrow \R$ is a continuous function  such that $b(0)\ge 0$,
\item[ii)] for $k\in K$, $\sigma_k:\R_+\rightarrow \R_+$ is a  continuous function such that $\sigma_k(0)=0$,
\item[iii)] for $i\in I$, let  $g_i:\R_+\times U_i\rightarrow \R$ be Borel functions such that  ${\sum}_{i\in I}g_i(x,u_i)+x\geq 0$ for $x\geq0$ and  $u_i\in U_i$,
\item[iv)] for $j\in J$, let  $h_j:\R_+\times V_j\rightarrow \R$ be Borel functions such that $h_j(0,v_j)=0$ and ${\sum}_{j\in J}h_j(x,v_j)+x\geq 0$ for $x>0$ and  $v_j\in V_j$.
\end{itemize}
For each $k\in K$,  let $B^{(k)}=(B_t^{(k)},t\geq 0)$ be  a standard Brownian motion.  We also let $(M_{i})_{i\in I}$ and $(N_j)_{j\in J}$ be  two sequences of Poisson random measures  such that  each  $M_i(\ud s,\ud u)$ and  $N_j(\ud s,\ud u) $ are defined on $\R_+\times U_i$ and $\R_+\times V_j$, respectively, and with  intensities  given by $\ud s\mu_i(\ud u)$ and $\ud s\nu_j(\ud v)$. We also suppose that $(B^{(k)})_{k\in K}$, $(M_{i})_{i\in I}$ and $(N_j)_{j\in J}$ are independent of each other. The compensated measure of $N_j$ is denoted by $\widetilde{N}_j$.

For each $i\in I$, let $W_i$ be a subset in $U_i$ such that $\mu_i(U_i\setminus W_i)<\infty$. For our purposes, we consider the following conditions on  the parameters $(b, (\sigma_k)_{k\in K},  (h_i)_{i\in I}, (g_j)_{j\in J})$:
\begin{itemize}
\item[a)] For each $n$, there is a positive constant $A_n$ such that
$$ \sum_{i\in I}\int_{W_i}|g_i(x,u_i)\wedge 1|\mu_i(\ud u_i)\leq A_n(1+x),\qquad \textrm{for every }\quad x\in[0,n].$$

\item[b)] Let $b(x)=b_1(x)-b_2(x)$,  where $b_1$ is a continuous function and $b_2$ is a non-decreasing continuous function. For each $n\geq 0$, there is a non-decreasing concave function $z\mapsto r_n(z)$ on $\R_+$ satisfying $\int_{0+} r_n(z)^{-1}\ud z=\infty$ and
$$|b_1(x)-b_1(y)|+\sum_{i\in I}\int_{W_i}|g_i(x,u_i)\wedge n-g_i(y,u_i)\wedge n|\mu_i(\ud u_i)\leq r_n (|x-y|)$$
for every $0\leq x,y\leq n$.

\item[c)] For each $n\ge 0$ and $(v_1,\cdots,v_l)\in\mathcal{V}$, the function $x\mapsto x+ {\sum}_{j\in J}h_j(x,v_j)\wedge n$ is non-decreasing and there is a positive constant $B_n$ such that for every $0\leq x,y\leq n$, 
$$\sum_{k\in K}|\sigma^2_k(x)-\sigma^2_k(y)|+\sum_{j\in J}\int_{V_j}\Big(|l_j(x,y,v_j)|\wedge l^2_j(x,y,v_j)\Big)\nu_j(\ud v_j)\leq B_n |x-y|$$
where $l_j(x,y,v_j)=h_j(x,v_j)\wedge n-h_j(y,v_j)\wedge n$.

\end{itemize}

A non-negative process $Z=(Z_t, t\ge 0)$ with c\`adl\`ag paths  is called a {\it solution} of 
\begin{equation}\label{sde}
\begin{split}
 Z_t=&Z_0+\int_0^t b(Z_s) \ud s+\sum_{k\in K}\int_0^t \sigma_k( Z_s)\ud B^{(k)}_s \\
 & +\sum_{i\in I}\int_0^t\int_{U_i}g_i(Z_{s-},u_i)M_i(\ud s,\ud u_i)+\sum_{j\in J}\int_0^t\int_{V_j}h_j(Z_{s-},v_j)\widetilde{N}_j(\ud s,\ud v_j),
\end{split}
\end{equation}
if it satisfies the stochastic differential equation a.s. for every $t\ge 0$. We say that $Z$ is a {\it strong solution} if, in addition, it is adapted  to the augmented natural filtration generated by $(B^{(k)})_{k\in K}$, $(M_{i})_{i\in I}$ and $(N_j)_{j\in J}$.
\begin{theorem}\label{Existencia Z}
Suppose that $(b, (\sigma_k)_{k\in K},  (h_i)_{i\in I}, (g_j)_{j\in J})$ are admissible parameters satisfying conditions a), b) and c).  Then, the stochastic differential equation (\ref{sde}) has a unique non-negative strong solution. The process $Z=(Z_t, t\geq 0)$ is a Markov process and its infinitesimal generator $\mathcal{L}$ satisfies, for every $f\in C^2_b(\overline{\R}_+),$\footnote{$\R_+=[0,\infty)$, $\overline{\R}_+=[0,\infty]$ and $C^2_b(\overline{\R}_+)=\{\mbox{twice differentiable functions such that } f(\infty)=0\}$}
 \begin{equation}
 \begin{split}\label{generado}
 \mathcal{L}f(x)=b(x)f'(x)+&\frac{1}{2}f''(x)\sum_{k\in K}\sigma_k^2(x)+\sum_{i\in I}\int_{U_i} \Big(f(x+g_i(x,u_i))-f(x)\Big)\mu_i(\ud u_i)\\
 &\hspace{.5 cm}+\sum_{j\in J}\int_{V_j} \Big(f(x+h_j(x,v_j))-f(x)-f'(x)h_j(x,v_j)\Big)\nu_j(\ud v_j). \\
  \end{split}
 \end{equation}
\end{theorem}
 
 \begin{proof} { We first prove that any solution of \eqref{sde} is non-negative.  Observe that one can extend the functions $b,\sigma_k,g_i,h_j$ to $\R$ in such a way that $b$ is continuous with $b(x)\geq 0$ for all $x\leq 0$, and $\sigma_k(x)=g_i(x,u_i)=h_j(x,v_j)=0$ for all $x\leq 0$ and $u_i\in U_i$, $v_j\in V_j$.  We now proceed as  in the proof of Proposition 2.1 in  \cite{FuLi}.  Let $\epsilon>0$ and  $\tau:=\inf\{t\geq 0: Z_t\leq -\epsilon\}$,  be such that  $\Prob{\tau<\infty}>0$. From  (iii) and (iv), it is clear that $Z$ cannot jump downwards implying that on the event $\{\tau<\infty\}$, $Z_{\tau}= Z_{\tau-}=-\epsilon$ and $\tau>\varsigma:=\sup\{s<\tau:Z_t\leq 0 \mbox{ for all } s\leq t\leq\tau\}$. Let $r\geq 0$ be such that $\Prob{\tau>r>\varsigma}>0$, thus the contradiction occurs by observing that $Z_{t\wedge \tau}$ is non-decreasing in $(r,\infty)$ and $Z_{r}>-\epsilon$. In other words any solution of \eqref{sde} is non-negative. }
 	
 	{ Now,} for each $i\in I$ and $j\in J$, let $\{W_i^\ell: \ell\in\N\}$ and $\{V_j^\ell: \ell\in\N\}$ be non-increasing sequences of Borel subsets of $W_i$ and $V_j$, such that  $\underset{\ell\in\N}\bigcup W_i^\ell=W_i$ and $\mu_i(W_i^\ell)<\infty$;  $\underset{\ell\in\N}\bigcup V_j^\ell=V_j$ and $\nu_j(V_j^\ell)<\infty$, respectably. { From Theorem IV.2.3 in \cite{IW}}, for each $n,\ell\in \N$, there is a non-negative weak solution to
 	\begin{equation}
 	\begin{split}\label{continuo}
 	Z_t=Z_0&+\int_0^t b(Z_s\wedge n) \ud s+\sum_{k\in K}\int_0^t \sigma_k( Z_s^{(n)}\wedge n)\ud B^{(k)}_s	\\
 	&\hspace{2cm}
	-\sum_{j\in J}\int_0^t\int_{V_j^\ell}\big(h_j(Z_{s}\wedge n,v_j)\wedge n\big) \mu_j(\ud v_j)\ud s .
 	\end{split}
 	\end{equation}
 By {H\"older's} inequality and  (c), the functions
 $x\mapsto \int_{V_j^\ell}\big(h_j(x\wedge n,v_j)\wedge n\big) \nu_j(\ud v_j)$
 are continuous for each $j\in J$ and $\ell\in \N$. Moreover $ b(x\wedge n)-\sum_{j\in J}\int_{V_j^\ell}\big(h_j(x\wedge n,v_j)\wedge n\big) \nu_j(\ud v_j)$ is the difference between the continuous function $b_1(x\wedge n)+\sum_{j\in J} (x\wedge n)\nu(V_j^\ell)$ and the non-decreasing continuous function $b_2(x\wedge n)+\sum_{j\in J}\int_{V_j^\ell}[(x\wedge n)+\big(h_j(x\wedge n,v_j)\wedge n\big)] \nu_j(\ud v_j)$. Then, by Lemma \ref{pathwise} (see the Appendix) the pathwise uniqueness holds for \eqref{continuo}, so the equation has a unique non-negative strong solution (see \cite[p.104]{situ}).

  For each $\ell, n\in \N$, by applying Lemma \ref{subconjuntow} (see the Appendix) to the admissible parameters $(b-\sum_{j\in J}\int_{V_j^\ell}\big(h_j(\cdot,v_j)\wedge n\big) \nu_j(\ud v_j),(\sigma_k)_{k\in K}, 0, (h_i)_{i\in I}\cup (g_j)_{j\in J}) $ we can add the integrals of $(h_i)_{i\in I}$ and $ (g_j)_{j\in J}$ in the spaces $(W_i^{\ell})_{i\in I}$ and $ (V_j^{\ell})_{j\in J}$, respectably. In other words, we deduce that for each $n,\ell\in \N$, there is a unique non-negative strong solution to 
 	\begin{equation*}
 	\begin{split}\label{z,n,m}
 	Z_t^{(n,\ell)}=Z_0&+\int_0^t b(Z_s^{(n,\ell)}\wedge n) \ud s+\sum_{k\in K}\int_0^t \sigma_k( Z_s^{(n,\ell)}\wedge n)\ud B^{(k)}_s \\
 	&\hspace{1cm}+\sum_{i\in I}\int_0^t\int_{W_i^\ell}\big(g_i(Z_{s-}^{(n,\ell)}\wedge n,u_i)\wedge n\big) M_i(\ud s,\ud u_i)\\
 	&\hspace{2cm}+\sum_{j\in J}\int_0^t\int_{V_j^\ell}\big(h_j(Z_{s-}^{(n,\ell)}\wedge n,v_j)\wedge n\big) \widetilde{N}_j(\ud s,\ud v_j).
 	\end{split}
 	\end{equation*}

 Let denote by $D(\R_{+},\R_{+})$ the space of c\`adl\`ag positive functions taking values on  $\R_+$.	For a fixed $n\in \N$, Lemma \ref{sequence znm} (see the Appendix) implies that the sequence $\{(Z_t^{(n,\ell)}, t\ge 0):\ell \in \N\}$ is tight in $D(\R_{+},\R_{+})$.
	Moreover,  by Lemma \ref{weak limit} (see the Appendix) any weak  limit point $(Z^{(n)}_t: t\geq 0)$  is a non-negative weak solution to
 	\begin{equation}
 	\begin{split}\label{sdenw}
 	Z_t^{(n)}=Z_0&+\int_0^t b(Z_s^{(n)}\wedge n) \ud s+\sum_{k\in K}\int_0^t \sigma_k( Z_s^{(n)}\wedge n)\ud B^{(k)}_s \\
 	&\hspace{1cm}+\sum_{i\in I}\int_0^t\int_{W_i}\big(g_i(Z_{s-}^{(n)}\wedge n,u_i)\wedge n\big) M_i(\ud s,\ud u_i)\\
 	&\hspace{2cm}+\sum_{j\in J}\int_0^t\int_{V_j}\big(h_j(Z_{s-}^{(n)}\wedge n,v_j)\wedge n\big) \widetilde{N}_j(\ud s,\ud v_j).
 	\end{split}
 	\end{equation}
 {Using again Lemma \ref{pathwise} (see the Appendix),  pathwise uniqueness holds for \eqref{sdenw}. }This guaranties that there is a unique non-negative strong  solution of (\ref{sdenw}). (see \cite[p.104]{situ}). Next, we apply Lemma \ref{subconjuntow} (see the Appendix) that allows us to replace the space $W_i$ by $U_i$ in the SDE (\ref{sdenw}). In other words we  deduce  that for $n\geq 0$ there is a unique non-negative strong  solution of   
 	\begin{equation}
 	\begin{split}\label{sden}
 	Z_t^{(n)}=Z_0&+\int_0^t b(Z_s^{(n)}\wedge n) \ud s+\sum_{k\in K}\int_0^t \sigma_k( Z_s^{(n)}\wedge n)\ud B^{(k)}_s \\
 	&\hspace{1cm}+\sum_{i\in I}\int_0^t\int_{{U_i}}\big(g_i(Z_{s-}^{(n)}\wedge n,u_i)\wedge n\big) M_i(\ud s,\ud u_i)\\
 	&\hspace{2cm}+\sum_{j\in J}\int_0^t\int_{V_j}\big(h_j(Z_{s-}^{(n)}\wedge n,v_j)\wedge n\big) \widetilde{N}_j(\ud s,\ud v_j).
 	\end{split}
 	\end{equation}

 Finally, we proceed to show that there is a unique non-negative strong  solution to the SDE (\ref{sde}). In order to do so, we first define $\tau_p=\inf\{t\geq 0: Z^{(p)}_t \geq q\}$,  for $p\geq q$, and then we prove that the sequence $(\tau_p, p\ge q)$ is non-decreasing and that $Z_t^{(p)}=Z_t^{(n)}$ for $p\leq n$ and $t< \tau_p$.  
 
 Observe that the trajectory $t\mapsto Z_t^{(p)}$ has no jumps larger than $p$ on the interval $[0,\tau_p)$. Since the Poisson random measures are independent, they do not jump simultaneously and therefore for each $i\in I$ and $j\in J$, we have for $0\le t<\tau_p$,
 \[
  g_i(Z_{t}^{(p)},u)\leq p\qquad  \textrm{and} \qquad h_j(Z_{t}^{(p)},v)\leq p, \quad u\in U_i,v\in V_j.
  \] 
  This implies  that $Z_t^{(p)}$ satisfies (\ref{sde}) on the interval $[0,\tau_p)$.  For $q\leq p\leq n,$ let $(Y^{(n)}_t, t\geq 0)$ be the strong solution to    
 \begin{align*}
 Y_t^{(n)}=&Z^{(p)}_{\tau_p-} +\int_0^t b(Y_s^{(n)}\wedge n) \ud s+\sum_{k\in K}\int_0^t \sigma_k( Y_s^{(n)}\wedge n)\ud B^{(k)}_{\tau_p+s} \\
    &\hspace{1.5cm}+\sum_{i\in I}\int_0^t\int_{U_i}\big(g_i(Y_{s-}^{(n)}\wedge n,u)\wedge n\big) M_i(\tau_p+\ud s,\ud u)\\
    &\hspace{3.5cm}+\sum_{j\in J}\int_0^t\int_{V_j}\big(h_j(Y_{s-}^{(n)}\wedge n,v)\wedge n\big) \widetilde{N}_j(\tau_p+\ud s,\ud v).
    \end{align*}
We  define $\tilde{Y}_t^{(n)}=Z_t^{(p)}$ for $0\leq t\leq \tau_p$ and $\tilde{Y}_t^{(n)}=Y_{t-\tau_p}^{(n)}$ for $t\geq \tau_p$. Note that $(\widetilde{Y}_t^{(n)}, t\ge 0)$ is  solution to (\ref{sden}). From the strong uniqueness, we deduce that  $Z_t^{(n)}=\widetilde{Y}_t^{(n)}$ for all $t\geq 0$. In particular, we have that $Z_t^{(n)}=Z_t^{(p)}<p$ for $0\leq t <\tau_p$. Consequently, the sequence $(\tau_p, p\ge q)$ is non-decreasing. 

Next, we define the process $Z=(Z_t,t\geq 0)$ as
 
 \[ 
 Z_t = \left\{ \begin{array}{ccl}
 Z_t^{(p)} & \textrm{if} & t< \tau_p, \\
 \infty & \textrm{if} & t\geq \underset{p\rightarrow\infty}{\lim}\tau_p. \end{array} \right.
 \] 
 \noindent It is not difficult to see that $Z$ is a weak solution to (\ref{sde}). In order to prove our result, we consider two solutions to (\ref{sde}), $Z'$ and $Z''$,  and consider $\tau'_p=\inf\{t\geq 0: Z'_t \geq p\}$, $\tau''_p=\inf\{t\geq 0: Z''_t \geq p\}$ and $\tau_p=\tau_p'\wedge \tau''_p$. Therefore  $Z'$ and $Z''$ satisfy (\ref{sden}) on $[0,\tau_p)$, implying that they are indistinguishable on $[0,\tau_p)$. If $\tau_{\infty}=\underset{p\rightarrow\infty}{\lim}\tau_p<\infty$, we have two possibilities either $Z'_{\tau_{\infty-}}=Z''_{\tau_{\infty-}}=\infty$ or one of them has a jump of infinity size at $\tau_{\infty}$. In the latter case, this jump comes from an atom of one of the Poisson random measures $(M_i)_{i\in I}$ or $(N_j)_{j\in J}$, so both processes have it. Since after this time both processes are equal to $\infty$,  we obtain that $Z'$ and $Z''$ are indistinguishable. In other words,  there is a unique strong solution to (\ref{sde}). (see \cite[p.104]{situ}). The strong Markov property follows since  there is a strong solution, the integrators are L\'evy processes and the integrand functions are not time dependent (see for instance \cite[Theorem V.32]{Protter}, where the Lipschitz property  is just needed to guarantee the existence and uniqueness of the  strong solution) and  by It\^o's formula it is easy to show that the strong solution $Z$ has generator given by (\ref{generado}).
 \end{proof}

\section{CBI-processes with competition in a L\'evy random environment}
 In this section, we construct  a branching model in continuous time and space that is affected by a random environment as the unique strong solution of a SDE that satisfies the conditions of Theorem 1. 
In this model, the random environment is driven by a general L\'evy process.
 
Our model is a natural extension of the  CB-processes in random environment studied by  Bansaye et al. \cite{Bapa},  in the case where the L\'evy process has paths of bounded variation,  and by Palau and Pardo \cite{PP}, in the case where the random environment is driven by a Brownian motion with drift.

In order to define  CBI-processes in a L\'evy random environment (CBILRE for short), we first introduce the objects that are involve on the {\it branching,  immigration} and {\it environmental} parts.
For the {\it branching part}, we introduce  $B^{(b)}=(B^{(b)}_t,t\geq 0)$ a standard Brownian motion and $N^{(b)}(\ud s,\ud z,\ud u)$    a Poisson random measure independent of $B^{(b)}$, with intensity $\ud s\Lambda(\ud z)\ud u$ where  $\Lambda(\ud z)=\mu(\ud z)\mathbf{1}_{(0,\infty)}(z)+q\delta_\infty(\ud z)$,   for $q\ge 0$. We denote by $\widetilde{N}^{(b)}$ for the compensated measure of $N^{(b)}$ and recall that the measure $\mu$ is concentrated on $(0,\infty)$ and satisfies
\[
\int_{(0,\infty)} (1\wedge z^2)\mu(\ud z)<\infty.
\]
The {\it immigration} term is given by a Poisson random measure $M(\ud s,\ud z)$ with intensity $\ud s\nu(\ud z)$ where the measure $\nu$ is supported in $(0,\infty)$ and satisfies
\[
\int_{(0,\infty)} (1\wedge z)\nu(\ud z)<\infty.
\]
Finally, for the {\it environmental} term, we introduce  $B^{(e)}=(B^{(e)}_t,t\geq 0)$  a standard Brownian motion and $N^{(e)}(\ud s, \ud z)$  a Poisson random measure in $\R_+\times \R$ independent of $B^{(e)}$ with intensity $\ud s\pi(\ud y)$, $\widetilde{N}^{(e)}$ its compensated version and $\pi$ is a measure concentrated on $\R\setminus\{0\}$ such that
 $$\int_{\R} (1\wedge z^2)\pi(\ud z)<\infty.$$
We will assume that all the objects involve in the branching, immigration and environmental terms  are mutually independent.

 A CBLRE with immigration and competition is defined as  the  solution of the SDE
\begin{equation}\label{csbplre}
\begin{split}
 Z_t=&Z_0+\int_0^t \big(\mathbf{d}+aZ_s\big) \ud s+\int_0^t \sqrt{2\gamma^2 Z_s}\ud B^{(b)}_s  \\
&-\int_0^t \beta(Z_s) \ud s+\int_0^t\int_{(0,\infty)}zM^{(im)}(\ud s,\ud z)+\int_0^t  Z_{s-}\ud S_s\\
&+\int_0^t\int_{(0,1)}\int_0^{Z_{s-}}z\widetilde{N}^{(b)}(\ud s,\ud z,\ud u)+\int_0^t\int_{[1,\infty)}\int_0^{Z_{s-}}zN^{(b)}(\ud s,\ud z,\ud u),
\end{split}
\end{equation}
where  $a\in \R$, $\mathbf{d}, \gamma\ge 0$, $\beta$ is a continuous non-decreasing function on $[0,\infty)$ with $\beta(0)=0$, 
 \begin{equation}\label{env}
 S_t=\alpha t+\sigma B^{(e)}_t+\int_0^t\int_{(-1,1)}(e^z-1) \widetilde{N}^{(e)}(\ud s,\ud z)+\int_0^t\int_{\R\setminus (-1,1)}(e^z-1)N^{(e)}(\ud s,\ud z),
 \end{equation}
 with $\alpha\in \mathbb{R}$ and $\sigma\geq 0$. 
\begin{corollary}\label{corollary1}
The stochastic differential equation (\ref{csbplre}) has a unique non-negative strong solution. The CBLRE $Z=(Z_t, t\geq 0)$ is a Markov process and its infinitesimal generator $\mathcal{A}$ satisfies, for every $f\in C^2_b(\overline{\R}_+),$
 \begin{equation*}
 \begin{split}\label{generador}
 \mathcal{A}f(x)&= \Big(ax+\alpha x-\beta(x)+\mathbf{d}\Big)f'(x)+\int_{(0,\infty)} \Big(f(x+z)-f(x)\Big)\nu(\ud z)\\
 &+\left(\gamma^2 x+\frac{\sigma^2}{2}x^2\right)f''(x)+x\int_{(0,\infty)} \Big(f(x+z)-f(x)-zf'(x)\mathbf{1}_{\{z<1\}}\Big)\Lambda(\ud z) \\
& \hspace{4cm}+\int_{\mathbb{R}} \Big(f(xe^z)-f(x)-x(e^z-1)f'(x)\mathbf{1}_{\{|z|<1\}}\Big)\pi(\ud z).
  \end{split}
 \end{equation*}
\end{corollary}
\begin{proof}
The proof of this result is a straightforward  application of Theorem 1. Take the set of index $K=J=\{1,2\}$ and  $I=\{1,2,3\}$; the  spaces 
\begin{align*}\label{spaces}
U_1=W_1&=[1,\infty)\times \mathbb{R}_+,&   U_2=&\R\setminus (-1,1), & W_2=(-\infty,-1],\\
 U_3= W_3&=\mathbb{R}_+,& V_1=&(0,1)\times \mathbb{R}_+,&  V_2=(-1,1),
\end{align*}
with associated Poisson  random measures $M_1=N^{(b)}, M_2=N^{(e)}, M_3=M^{(im)}, N_1=N^{(b)}$ and $N_2=N^{(e)}$, respectively; and standard Brownian motions $B^{(1)}=B^{(b)}$ and $B^{(2)}=B^{(e)}$.  We  also take the functions
\begin{align*}
 b(x)&=ax-\beta(x)+\mathbf{d},& \sigma_1(x)&=\sqrt{2\gamma^2 x},& \sigma_2(x)&=\sigma x, &\\
   g_1(x, z, u)&=z\mathbf{1}_{\{u\le x\}},& g_2(x, z)&=x(e^z-1),& g_3(x,z)&=z,&\\
   h_1(x, z, u)&=z\mathbf{1}_{\{u\le x\}},&  h_2(x, z)&=x(e^z-1),&
\end{align*}
which are admissible and verify conditions a), b) and c). 
\end{proof}

Similarly to the cases of Bansaye et al. \cite{Bapa} and  Palau and Pardo \cite{PP}, we can compute the Laplace transform of a reweighted version of $Z$ given the environment and under the assumption that $\beta\equiv 0$. In order to do so, we assume that $q=0$, $\beta\equiv 0$ and 
\[
\hspace{5cm} \int_{[1,\infty)}x\mu(\ud x)<\infty. \hspace{6.8cm}(\mathbf{H}) 
 \]

 Recall that the associated branching mechanism $\psi$ satisfies 
\begin{equation*}\label{lk}
\psi(\lambda)=-a\lambda+\gamma^2
\lambda^2+\int_{(0,\infty)}\big(e^{-\lambda x}-1+\lambda x{\mathbf 1}_{\{x<1\}}\big)\mu(\ud x),\qquad \lambda\geq 0,
\end{equation*}
and observe that  from our assumption,  $|\psi^\prime(0+)|<\infty$ and
\[
\psi^\prime(0+)=-a-\int_{[1,\infty)} x\mu(\ud x).
\]
We also recall that the  immigration mechanism is given by 
$$\phi(u)=\mathtt{d}u+\int_0^{\infty} (1-e^{-u t})\nu(\ud t),\qquad u\geq 0.$$ 
For the sequel, we define the auxiliary process
  \begin{equation}\label{auxiliary}
 K_t=\mathbf{m} t+\sigma B^{(e)}_t+\int_0^t\underset{(-1,1)}{\int}v\widetilde{N}^{(e)}(\ud s,\ud v)+\int_0^t\int_{\R\setminus (-1,1)}v N^{(e)}(\ud s,\ud v),
 \end{equation}
where 
\[
\mathbf{m}=\alpha-\psi^\prime(0+)-\frac{\sigma^2}{2}-\underset{(-1,1)}{\int}(e^v-1-v)\pi(\ud v).
\]
 It is important to note that when $\psi\equiv 0$, conditionally on the environment $K$, the process $Z$ satisfies the branching property. The proof {of this fact  is the same as in the Brownian environment case. (see Theorem 1 in  \cite{PP}). }

\begin{proposition}\label{Prop1}
Suppose that $(\mathbf{H})$ holds.  Then for every $z,\lambda,t>0$, a.s.
\begin{equation}\label{bplaplacek}
\begin{split}
\Expconx{z}{\Expo{-\lambda Z_t e^{-K_t}}}{K}=\Expo{-zv_t(0,\lambda,K)-\int_0^t\phi\Big(v_t(r,\lambda,K )e^{-K_r}\Big)\ud r},
\end{split}
\end{equation}
 where for every $t,\lambda\geq 0$, the function  $(v_t(s,\lambda, K), s\leq t)$ is the a.s. unique  solution of the backward differential equation
 \begin{align}\label{bpbackward}
 \frac{\partial}{\partial s}v_t(s,\lambda, K)=e^{K_s}\psi_0(v_t(s,\lambda, K)e^{-K_s}),\qquad v_t(t,\lambda, K)=\lambda,
 \end{align}
and
 $$\psi_0(\lambda)=\psi(\lambda)-\lambda\psi'(0)=\gamma^2
\lambda^2+\int_{(0,\infty)}\big(e^{-\lambda x}-1+\lambda x\big)\mu(\ud x),\qquad \lambda\geq 0.$$
 \end{proposition}
 
  \begin{proof} The first part of the proof follows  similar arguments as those used in Bansaye et al. \cite{Bapa}. The main problem in proving our result is finding  the a.s. unique  solution of the backward differential equation (\ref{bpbackward}) in the general case. In order to do so, we  need an approximation technique based on the L\'evy-It\^o decomposition of the L\'evy process $K$.  The proof of the latter can be found in the Appendix in Lemma \ref{existencev}.
  
  For sake of completeness, we remain the main steps of the proof which are similar as those used in \cite{Bapa}. We first define $\widetilde{Z}_{t}=Z_{t}e^{-K_{t}}$, for $t\ge 0$, and choose 
  \[
  F(s,x)=\Expo{-xv_t(s,\lambda,K )-\int_s^t\phi(v_t(r,\lambda,K )e^{-K_r})\ud r}, \qquad s,x\geq 0,
  \]
   where $v_t(s,\lambda,K)$ is differentiable with respect to the variable $s$, non-negative and such that  $v_t(t,\lambda, K)=\lambda$ for all $\lambda\geq 0$. We observe,  conditioned on $K$,  that  $(F(s,\widetilde{Z}_s),s\leq t)$ is a martingale (using It\^o's formula) if and only if 
   \[
   \begin{split}
   \frac{\partial}{\partial s}v_t(s,\lambda, K)=&\gamma^2v_t(s,\lambda, K)^2e^{-K_s}\\
&+ e^{K_s}\int_{0}^{\infty}\left(e^{-e^{-K_s} v_t(s,\lambda, K)z}-1+e^{-K_s}v_t(s,\lambda, K)z\right)\mu(\ud z),
   \end{split}
   \]
   which is equivalent that $v_t(s,\lambda, K)$ solves (\ref{bpbackward}). Providing that $v_t(s,\lambda, K)$ exist a.s., we see that   the process 
    $\left(\Expo{-\widetilde{Z}_{s}v_{t}(s,\lambda, K)}, 0\leq s\leq t\right)$ conditioned on $K$ is a martingale, and hence 
$$\Expconx{z}{\Expo{-\lambda \widetilde{Z_t}}}{K}=\Expo{-zv_t(0,\lambda,K)-\int_0^t\phi(v_t(r,\lambda,K )e^{-K_r})\ud r}.$$
   \end{proof}
   
{
It is important to note that if   $|\psi^\prime(0+)|=\infty$,  the  auxiliary process can be taken as 
  	\begin{equation*}
 {K}^{(0)}_t=\mathbf{n} t+\sigma B^{(e)}_t+\int_0^t\int_{(-1,1)}v\widetilde{N}^{(e)}(\ud s,\ud v)+\int_0^t\int_{\R\setminus (-1,1)}v N^{(e)}(\ud s,\ud v),
  	\end{equation*}
  	where $\mathbf{n}=\alpha-\frac{\sigma^2}{2}-\underset{(-1,1)}{\int}(e^v-1-v)\pi(\ud v).$
  	Following  the same arguments as in the previous proposition and  replacing $K$ with ${K}^{(0)}$, one can deduce that $v_t(s,\lambda, {K}^{(0)})$ is the unique solution to  the backward differential equation
  	\begin{equation}\label{bpbackwardtil}
  	\frac{\partial}{\partial s}v_t(s,\lambda, {K}^{(0)})=e^{{K}^{(0)}_s}\psi(v_t(s,\lambda,{K}^{(0)})e^{-{K}^{(0)}_s}),
  	\end{equation}
  	whenever this equation has a solution.}
   
  {The latter observation allows us to compute explicitly the Laplace exponent of  the so called Neveu case. Recall that the  Neveu branching process in a L\'evy random environment has branching mechanism given by 
   $$\psi(u)=u\log(u)=cu+\int_{(0,\infty)}\big(e^{-u x}-1+u x{\mathbf 1}_{\{x<1\}}\big)x^{-2}\ud x,\qquad u>0, $$ where $c\in\R$ is a suitable constant. In this particular case, the backward differential  equation (\ref{bpbackwardtil}) satisfies
   $$\frac{\partial}{\partial s}v_t(s,\lambda, {K}^{(0)})=v_t(s,\lambda, \delta)\log(e^{-{K}^{(0)}_s} v_t(s,\lambda, {K}^{(0)}).$$}
  { The above equation has a solution and is given by 
   $$v_t(s, \lambda, {K}^{(0)})=\Expo{e^{s}\left(\int_s^t e^{-u}{K}^{(0)}_u \ud u+\log(\lambda)e^{-t}\right)}.$$}
{Hence,  for all $z,\lambda,t>0$, we deduce
   \begin{align}\label{bplaplacekneveu}
   \e_z\Big[\exp\Big\{-\lambda Z_t e^{-{K}^{(0)}_t}\Big\}\Big| {K}^{(0)}\Big]=\Expo{-z\lambda^{e^{-t}}\Expo{\int_0^t e^{-s}{K}^{(0)}_s \ud s}}\qquad \textrm{a.s.}
   \end{align}
 By taking limits as $\lambda$ goes to $\infty$ in the previous expression and then taking expectation, we obtain $\mathbb{P}_z(Z_t>0)=1$, for $t\geq 0.$}
 
{By integration by parts, we deduce
 \[
 \int_0^t e^{-s}{K}^{(0)}_s \ud s=-e^{-t}{K}^{(0)}_t + \int_0^t e^{-s}\ud {K}^{(0)}_s.
 \]
Under the assumption $\mathbb{E}[|{K}^{(0)}_1|]<\infty$, we deduce from Theorem 17.5 in Sato \cite{sa}, that the r.v. $ \int_0^\infty e^{-s}{K}^{(0)}_s \ud s=\int_0^{\infty} e^{-s}\ud {K}^{(0)}_s$   is self-decomposable   with characteristic exponent given
  by $\psi(\lambda)=\int_0^{\infty} \psi_{{K}^{(0)}}(\lambda e^{-s})\ud s,$ where $\psi_{{K}^{(0)}}$ denotes the characteristic exponent of ${K}^{(0)}$. Therefore, if we  take limits as $t\uparrow \infty$ in (\ref{bplaplacekneveu}), we observe that for $z,\lambda>0$
 \[
 \e_z\Big[\exp\Big\{-\lambda  \lim_{t\to\infty} Z_t e^{-{K}^{(0)}_t}\Big\}\Big | {K}^{(0)}\Big] =\Expo{-z\Expo{\int_0^\infty e^{-s} \ud {K}^{(0)}_s}}.
 \]
 Since the right-hand side of the above identity does not depend on $\lambda$, this implies that 
 \[
 \p_z\Big(\lim_{t\to\infty} Z_t e^{-{K}^{(0)}_t}=0\Big)=\e\left[\Expo{-z\Expo{\int_0^\infty e^{-s} \ud {K}^{(0)}_s}}\right].
 \]
 In conclusion, the Neveu process in L\'evy environment survives a.s. but when the L\'evy process ${K}^{(0)}$ does not drift to  $\infty$, the extinction probability is given by the previous expression.
}
 
{Another interesting example is  the self-similar CB-processes in a L\'evy random environment.  In this example the  branching mechanism is taken as follows
$$\psi(\lambda)=c \lambda^{\alpha}, \qquad \lambda \geq 0,$$
for some $\alpha\in(0,1)\cup(1,2]$  and $c_\alpha\in   \R$ such that $c_\alpha(\alpha-1)>0$. If $\alpha=2$, we observe that   $\mu(0,\infty)=0$ and $c=\gamma^2$.  In the case  $\alpha\in (0,1)\cup(1,2)$, the process $Z$ satisfies the SDE
\begin{align*}
Z_t =& Z_0 +\int_0^t  Z_{s-}\ud S_s+\int_0^t\int_0^{\infty}\int_0^{Z_{s-}}z\hat{N}(\ud s,\ud z,\ud u), 
\end{align*}
 where $N$ is an independent Poisson random measure with intensity 
\[
\frac{c_\alpha \alpha(\alpha-1)}{\Gamma(2-\alpha)}\frac{1}{z^{1+ \alpha}}\ud s \ud z\ud u,
\]
$\widetilde{N}$ is its  compensated version and $\hat{N}=N\Ind{\alpha\in(0,1)}+\widetilde{N}\Ind{\alpha\in(1,2)}$.
Note  that
\[
\psi^\prime(0+) =\left\{ \begin{array}{ll}
-\infty  &\textrm{ if $\alpha \in(0,1),$}\\
0 & \textrm{ if $\alpha \in (1,2].$}
\end{array} \right .
\]
Hence, when $\alpha\in(1,2]$, we have  $K_t={K}^{(0)}_t$, for $t\ge 0$. We use the backward differential  equation (\ref{bpbackwardtil}) and observe that it satisfies
$$\frac{\partial}{\partial s}v_t(s,\lambda, {K}^{(0)})=c_\alpha  v^{\alpha}_t(s,\lambda, {K}^{(0)}) e^{-(\alpha-1) {K}^{(0)}_s}.$$
Assuming that  $v_t(t,\lambda, K^{(0)})=\lambda$, we can solve the above equation  and after some straightforward computations, we get
$$v_t(s, \lambda,K^{(0)})=\left(\lambda^{1-\alpha}+(\alpha-1) c_\alpha \int_s^t e^{-(\alpha-1) {K}_u^{(0)}} \ud u\right)^{-1/(\alpha-1)}.$$
Hence, from (\ref{bplaplacek}) the following  identity holds a.s.
\[
\e_z\Big[\exp\Big\{-\lambda Z_t e^{-K_t^{(0)}}\Big\}\Big| K^{(0)}\Big]=\Expo{-z\left(\lambda^{1-\alpha}+(\alpha-1) c_\alpha\int_0^t e^{-(\alpha-1) K_u^{(0)}} \ud u\right)^{-1/(\alpha-1)}}.
\]
Observe  that  the probabilities of survival and  non-explosion can be determined explicitly in terms of the exponential functional  $\int_0^t e^{-(\alpha-1) K^{(0)}_u} \ud u$,  by taking $\lambda$ goes to $\infty$ or $\lambda$ goes to $0$, respectively. In other words, for all $z>0$
\begin{align*}
\mathbb{P}_z\Big( Z_t >0\Big|K^{(0)}\Big)=1-\mathbf{1}_{\{\alpha>1\}}
\exp\left\{-z\left((\alpha-1) c_\alpha \int_0^t e^{-(\alpha-1)K^{(0)}_u} \ud u\right)^{-1/(\alpha-1)}\right\},\qquad \textrm{a.s.},
\end{align*}
and
\begin{align*}
\mathbb{P}_z\Big( Z_t < \infty\Big|K^{(0)}\Big)=
\mathbf{1}_{\{\alpha>1\}}+
\mathbf{1}_{\{\alpha<1\}}\exp\left\{-z\left((\alpha-1) c_\alpha\int_0^t e^{-(\alpha-1)K^{(0)}_u} \ud u\right)^{-1/(\alpha-1)}\right\},\quad \textrm{a.s.}
\end{align*}
The asymptotic behaviour of these probabilities has been computed recently by Palau et al. \cite{PPS}, using a fine study of the negative moments of exponential functional of L\'evy processes. In particular, and similarly to the results obtained by Bansaye et al. \cite{Bapa} and Palau and Pardo \cite{PP}, the authors in \cite{PPS} obtained five different regimes for the probability of survival when $\alpha\in(1,2]$ and three different regimes for the probability of non-explosion when $\alpha\in(0,1)$. Both depend on the characteristic of the L\'evy process $K^{(0)}$.}

\subsection{Long term behaviour of CB-processes in a L\'evy random environment.}
In the sequel,  we exclude from our model the competition mechanism  $\beta$ and the immigration term  $M^{(im)}$. In this section, we are interested in determining the long term behaviour of CB-processes in a L\'evy random environment (CBLRE for short). Our methodology follows  similar arguments as those  used in Proposition 2 in \cite{PP} and Corollary 2 in \cite{Bapa}.

Let $\Psi_K$ denotes the characteristic exponent of the L\'evy process $K$, i.e. 
\[
\Psi_K(\theta)=-\log\mathbb{E}[e^{i\theta K_1}]\qquad \qquad \textrm{for}\quad\theta\in \mathbb{R}.
\]
We also introduce the functions $\Phi(\lambda)=\lambda^{-1}\psi_0(\lambda)$, for $\lambda\ge 0$, and  
\[
A(x)=\mathbf{m}+\pi((1,\infty))+\int_{1}^x \pi((y,\infty))\ud y, \qquad \textrm{for}\quad x>0.
\]  

\begin{proposition} Assume that $(\mathbf{H})$ holds.  
Let $(Z_t, t\ge 0)$ be a CBLRE with  branching mechanism given by $\psi$, and $z>0$
\begin{enumerate}
\item[i)] If the process $K$ drifts to $-\infty$, then $\mathbb{P}_z\Big(\underset{t\rightarrow\infty}{\lim} Z_t=0\Big|K\Big)=1$, a.s.
\item[ii)] If the process $K$ oscillates, then $\mathbb{P}_z\left(\underset{t\rightarrow\infty}{\liminf} Z_t=0\Big| K\right)=1$, a.s. Moreover if $\gamma>0$  then 
\[
\mathbb{P}_z\Big(\underset{t\rightarrow\infty}{\lim}Z_t=0\Big|K\Big)=1,  \textrm{a.s.}
\]
\item[iii)] If  the process $K$ drifts to $+\infty$, so that $A(x)>0$ for all $x$ larger than some $a>0$. Then if
\begin{equation}\label{intcond}
\int_{(a,\infty)}\frac{x}{A(x)}\big|\ud \Phi( e^{-x})\big|<\infty,
\end{equation}
we have $\mathbb{P}_z\Big(\underset{t\rightarrow\infty}{\liminf}Z_t>0\Big|K\Big)>0 $ a.s.,  and there exists a non-negative finite r.v. $W$ such that
$$Z_te^{-K_t}\underset{t\rightarrow \infty}{\longrightarrow}W,\ \textrm{a.s} \qquad \textrm{and}\qquad \big\{W=0\big\}=\Big\{\lim_{t\rightarrow\infty}Z_t=0\Big\}.$$
In particular, if $0< \mathbb{E}[K_1]<\infty$ then the above integral condition is equivalent to
  \[
 \int^\infty x\ln x\, \mu({\rm d} x)<\infty.
\]  

\end{enumerate}
\end{proposition}
\begin{proof} Parts (i) and (ii) follow from the  same arguments used in the proof of Proposition 2 in \cite{PP}, so we skip their proofs.

Now, we prove part (iii).  We first observe that  ${v}_t(\cdot,\lambda, K)$, the a.s. solution to the backward differential equation (\ref{bpbackward}), is non-decreasing on $[0,t]$ (since $\psi_0$ is positive). Thus for all $s\in[0,t]$, ${v}_t(s,\lambda,K)\leq \lambda$.  Observe that  $\Phi(0)=\psi'_0(0+)=0$ and since $\psi_0$ is convex, we deduce that $\Phi$ is increasing.   Hence
\begin{align*}
\frac{\partial}{\partial s}{v}_t(s,\lambda, K)&={v}_t(s,\lambda,K)\Phi({v}_t(s,\lambda, K)e^{-K_s} )\leq {v}_t(s,\lambda,K)\Phi(\lambda e^{-K_s}).
\end{align*}
Therefore, for every $s\le t$, we have
\[
v_t(s, \lambda, K)\ge \lambda \exp\left\{-\int_{s}^t\Phi(\lambda e^{-K_s}){\rm d} s \right\}.
\]
In particular, 
\[
\lim_{t\to\infty} v_t(0, \lambda, K)\ge \lambda \exp\left\{-\int_{0}^\infty\Phi(\lambda e^{-K_s}){\rm d} s \right\}.
\]
If the integral on the right-hand side is a.s. finite, then
\[
\lim_{t\to\infty} v_t(0, \lambda, K)\ge \lambda \exp\left\{-\int_{0}^\infty\Phi(\lambda e^{-K_s}){\rm d} s \right\}>0, \qquad \textrm{a.s.,}
\]
 implying 
 \[
 \mathbb{E}_z\Big[e^{-\lambda W}\Big| K\Big]\le \exp\left\{-z\  \lambda \exp\left\{-\int_{0}^\infty\Phi(\lambda e^{-K_s}){\rm d} s \right\}\right\}<1,\qquad \textrm{a.s.}
 \]
 and in particular $\mathbb{P}_z\Big(\underset{t\rightarrow\infty}{\liminf}Z_t>0\Big|K\Big)>0 $ a.s. Next, we use Lemma 20 in \cite{Bapa} and the branching property of $Z$, to deduce
$$\{W=0\}=\Big\{\underset{t\rightarrow\infty}{\lim}Z_t=0\Big\}.$$
In order to finish our proof, we show that the integral condition  (\ref{intcond}) implies 
\[
\int_{0}^\infty\Phi(\lambda e^{-K_s}){\rm d} s <\infty \quad \textrm{a.s.}
 \]
 We first introduce $\varsigma=\sup\{t\ge 0: K_t\le 0\}$  and observe
 \begin{equation*}\label{intfin}
 \int_{0}^\infty\Phi( \lambda e^{-K_s}){\rm d} s=\int_{0}^\varsigma\Phi(\lambda e^{-K_s}){\rm d} s +\int_{\varsigma}^\infty\Phi(\lambda e^{-K_s}){\rm d} s
 \end{equation*}
Since $\varsigma<\infty$ a.s., the first integral of the right-hand side is a.s. finite. For the second integral, we use Theorem 1 in Erickson and Maller \cite{EM} which assures us that 
\[
\int_{\varsigma}^\infty\Phi(\lambda e^{-K_s}){\rm d} s<\infty,\qquad \textrm{a.s.},
\]
 if the integral condition  (\ref{intcond})  holds.

Finally, we assume that $0\le \mathbb{E}[K_1]<\infty$ and observe that $\lim_{x\to \infty} A(x)$ is finite. In particular, this implies that the  integral condition  (\ref{intcond}) is equivalent to 
\[
\int_{0}^\infty\Phi(\lambda e^{-y}){\rm d} y<\infty.
\]
On the other hand, we have 
\[
\begin{split}
\int_0^\infty\Phi(\lambda e^{-y}){\rm d} y &=\int_{0}^\lambda\frac{\Phi(\theta)}{\theta}{\rm d} \theta\\
&=\gamma^2\lambda +\int_0^\lambda \frac{{\rm d} \theta}{\theta^2}\int_{(0,\infty)} (e^{-\theta x}-1+\theta x)\mu ({\rm d}x)\\
&=\gamma^2\lambda +\int_{(0,\infty)}\mu ({\rm d}x)\int_0^\lambda (e^{-\theta x}-1+\theta x)\frac{{\rm d} \theta}{\theta^2} \\
&=\gamma^2\lambda +\int_{(0,\infty)} x\left(\int_0^{\lambda  x}(e^{-y}-1+y)\frac{{\rm d} y}{y^2} \right )\mu ({\rm d}x).
\end{split}
\]
Since the function
\[
g_\lambda(x)=\int_0^{\lambda  x}(e^{-y}-1+y)\frac{{\rm d} y}{y^2},
\]
is equivalent to $\lambda x/2$ as $x\to 0$ and equivalent to  $\ln x$ as $x\to \infty$, we deduce that 
\[
\int_{0}^\infty\Phi(\lambda e^{-y}){\rm d} y<\infty \qquad  \textrm{if and only if } \qquad \int^\infty x\ln x \mu({\rm d}x )<\infty.
\]
\end{proof}
Now, we  derive a central limit theorem in the supercritical regime which follows from Theorem 3.5 in Doney and Maller \cite{DM} and similar arguments as those used in Corollary 3 in \cite{Bapa}, so we skip its proof.

For $x>0$, let
\[
T(x)=\pi((x,\infty))+\pi((-\infty, -x))\qquad \textrm{and}\qquad U(x)=\sigma^2+\int_0^x yT(y)\ud y
\]
\begin{corollary}\label{TCL}
Assume that $K$ drifts to $+\infty$, $T(x)>0$ for all $x>0$, and (\ref{intcond}) is satisfied. There are two measurable functions $a(t), b(t)> 0$ such that
, conditionally on $\{ W>0 \}$, 
$$\frac{\log (Z_t) -a(t)}{ b(t)} \xrightarrow[t\to \infty]{d} {\mathcal{N}}(0,1),$$
if and only if 
\[
\frac{U(x)}{x^2T(x)}\to \infty \qquad \textrm{as }\quad x\to \infty, 
\]
where $\xrightarrow{d}$ means convergence in distribution and ${\mathcal{N}}(0,1)$ denotes a centered Gaussian random variable with variance equals 1. 
\end{corollary}

It is important to note that  if $\int_{\{|x|>1\}} x^2\pi(\ud x)<\infty$, then for all $t>0$,
\[
a(t):=\left(\mathbf{m}+\int_{\{| x |\ge 1\}} x \pi(\ud {x})\right) t\qquad \textrm{and} \qquad b^2(t):=\left(\sigma^2+\int_{\mathbb{R}}x^2\pi(\ud {x})\right) t,
\]
 which is similar to the result obtained in Corollary 3 in \cite{Bapa}.

\subsection{Population model with competition in a L\'evy random environment}
We now study an extension of the competition model given in Evans et al. \cite{EHS}. In this model, we exclude the immigration term and take the branching and competition mechanisms as follows
\[
\beta(x)=kx^2\qquad \textrm{and }\qquad \psi(\lambda)=a\lambda \qquad \mbox{for } x,\lambda\geq 0,
\]
where $k$ is a positive constant. Hence, we define  a  branching process  in a L\'evy random environment process $(Z_t,t \geq 0)$   as the solution of the SDE
\begin{align}\label{evans}
 Z_t=Z_0+\int_0^t Z_s(a-kZ_s)\ud s+\int_0^t  Z_{s-}\ud S_s 
\end{align}
where the environment is given by the L\'evy process given in (\ref{env}).

From Corollary \ref{corollary1},  there is a unique non negative strong solution of (\ref{evans}) satisfying the  Markov property. Moreover, we have the following result that in particular says that the process $Z$ is the inverse of a generalised Ornstein-Uhlenbeck process.
\begin{proposition}
Suposse that $(Z_t,t\geq 0)$ is the unique strong solution of (\ref{evans}). Then, it satisfies
\begin{equation}\label{equationevans}
Z_t=\frac{Z_0e^{K_t}}{1+kZ_0\displaystyle\int_0^t e^{K_s}\ud s},\qquad t\ge 0, 
\end{equation}
where $K$ is the L\'evy process defined in (\ref{auxiliary}). Moreover, if $Z_0=z>0$ then, $Z_t>0$ for all $t\geq 0$ a.s. and it has the following asymptotic behaviour:
\begin{enumerate}
\item[i)] If the process $K$ drifts to $-\infty$, then $\lim_{t\rightarrow\infty}Z_t=0$ a.s.
\item[ii)] If the process $K$ oscillates, then $\liminf_{t\rightarrow\infty}Z_t=0$ a.s.
\item[iii)] If the process $K$ drifts to $\infty$, then $(Z_t, t\ge 0)$ has a stationary distribution whose density satisfies for $z>0$,
\[
\mathbb{P}_z(Z_\infty\in \ud x)=h\left(\frac{1}{kx}\right)\frac{\ud x}{x^2}, \qquad x> 0,
\]
where
\[
\int_t^\infty h(x)\ud x=\int_\mathbb{R} h(te^{-y})U(\ud y),\qquad \textrm{a.e. $t$ on }(0,\infty), 
\]
and $U$ denotes the potential measure associated to $K$, i.e.
\[
U(\ud x)=\int_0^\infty \mathbb{P}(K_s\in \ud x)\ud s,\qquad x\in\R.
\]
Moreover, if $0<\Exp{K_1}<\infty$, then
$$\lim_{t\rightarrow\infty}\frac{1}{t}\int_0^t Z_s\ud s=\frac{1}{k}\Exp{K_1}, \qquad \textrm{a.s.}$$
\end{enumerate}
\end{proposition}

\begin{proof}By It\^o's formula, we see that the process $Z$ satisfies (\ref{equationevans}). Moreover, since the L\'evy process $K$ has infinite lifetime, then we necessarilly have $Z_t>0$ a.s.  

Now in order to  describe the asymptotic behaviour of $Z$ we recall the following result of  Bertoin and Yor \cite{beryor} on exponential functionals of L\'evy processes,
\begin{align}\label{f exponencial}
\int_0^\infty e^{K_s}\ud s<\infty \quad a.s.  \qquad\mbox{ if and only if } \qquad K \mbox{ drifts to }-\infty.
\end{align}
Therefore part (i) follows directly from (\ref{f exponencial}). Next, we prove part (ii). Assume that the process $K$ oscillates. On the one hand, we have 
$$ Z_t=\frac{Z_0}{e^{-K_t}+kZ_0e^{-K_t}\displaystyle\int_0^t e^{K_s}\ud s}\leq \frac{1}{ke^{-K_t}\displaystyle\int_0^t e^{K_s}\ud s}.$$
On the other hand, from the duality Lemma (see for instance Lemma 3.4 in \cite{Kyp}) we deduce
\[
\left(K_t,e^{-K_t}\int_0^t e^{K_s}\ud s\right) \qquad \textrm{is equal in law to}\qquad \left(K_t,\int_0^t e^{-K_s}\ud s\right). 
\]
From (\ref{f exponencial}) and our assumption, we have that the exponential functional of $K$ goes to  $\infty$ as $t\rightarrow \infty$. This implies that $\lim_{t\rightarrow\infty}Z_t=0$ in distribution and therefore,
\[
\liminf_{t\rightarrow\infty}Z_t=0,\qquad \textrm{a.s.}
\] 
 Finally, we assume that the process $K$ drifts to $\infty$.  Then, form  the previous observation, $Z_t$ is equal in law to 
$$Z_t\overset{d}{=}\frac{Z_0}{e^{-K_t}+kZ_0e^{-K_t}\displaystyle\int_0^t e^{K_s}\ud s}.$$
Using (\ref{f exponencial}), we have that $Z_t$ converges in distribution to 
$$\left(k\int_0^{\infty}e^{-K_s}\ud s\right)^{-1}.$$
The form of the density follows from Theorem 1 of Arista and Rivero \cite{AR}.

We finish our proof by observing that 
$$\int_0^t Z_s\ud s=\frac{1}{k}\ln\left(1+kZ_0\int_0^te^{K_s}\ud s \right).$$
Therefore if $0<\Exp{K_1}<\infty$ a simple application of the law of large numbers allow us to deduce
$$\lim_{t\rightarrow\infty}\frac{1}{t}\int_0^t Z_s\ud s=\lim_{t\rightarrow\infty}\frac{1}{kt}\ln\left(\int_0^te^{K_s}\ud s \right)=\frac{1}{k}\Exp{K_1}, \quad \textrm{a.s.}.$$
This completes the proof.
\end{proof}
The asymptotic behaviour of the positive moments of $Z_t$ has been studied in  Palau et al. \cite{PPS} using a fine study of the negative moments of exponential functional of L\'evy processes. In particular four  different regimes appears that depends on the characteristic of the L\'evy process $K$. 

We finish this section with two important observations in two particular cases.  We first assume that the process $K$ drifts to $+\infty$ and that satisfies
\[
\int_{[1,\infty)} e^{qx}\pi(\ud x)<\infty\qquad \textrm{for every}\quad q>0,
\]
i.e. that has exponential moments of all positive orders. In this situation, the characteristic exponent $\Psi_k$ has an analytic extension to the half-plane with negative imaginary part, and one has 
\[
\mathbb{E}[e^{qK_t}]=e^{t\psi_K(q)}<\infty, \qquad t,q\ge 0,
\]
where $\psi_K(q)=-\Psi_K(-iq)$ for $q\ge 0$. Hence, according to Theorem 3 in Bertoin and Yor \cite{beryor} the stationary distribution has positive moments and satisfies, for $z>0$ and $n\ge 1$,
\[
\mathbb{E}_z\Big[Z^n_\infty\Big]=\psi^\prime_K(0+)\frac{\psi_{K}(1)\cdots \psi_{K}(n-1)}{(n-1)!}.
\]
Finally,  we assume that the process $K$ drifts to $-\infty$ and has no negative jumps. Observe that the process $Z$ inherited the latter property and we let $Z_0=z>0$. Under this assumption, we can compute the Laplace transform of the first passage time from below a level $z>b> 0$ of the process $Z$, i.e.
\[
\sigma_b=\inf\{s\ge 0: Z_s\le b\}.
\]
In this situation, the characteristic exponent $\Psi_k$ has an analytic extension to the half-plane with positive imaginary part, and one has 
\[
\mathbb{E}[e^{-qK_t}]=e^{t\hat{\psi}_K(q)}<\infty, \qquad t,q\ge 0,
\]
where $\hat{\psi}_K(q)=-\Psi_K(iq)$ for $q\ge 0$.
Define, for all $t\geq 0$, $\mathcal{F}_t = \sigma(K_s: s\leq t)$ and consider the Esscher transform
\begin{equation*}
\frac{\ud \mathbb{P}^{\kappa(\lambda)}}{\ud \mathbb{P}}\bigg|_{\mathcal{F}_t}=e^{-\kappa(\lambda) K_t-\lambda t},\qquad \textrm{for } \lambda\ge 0,
\label{esscher}
\end{equation*}
where $\kappa(\lambda)$ is the largest solution to $\hat{\psi}_K(u)=\lambda$.
Under $\mathbb{P}^{\kappa(\lambda)}$, the process $K$ is still a spectrally positive and its Laplace exponent, $\hat{\psi}_{\kappa(\lambda)}$ satisfies the relation
\[
\hat{\psi}_{\kappa(\lambda)}(u)=\hat{\psi}_K(\kappa(\lambda)+u)-\lambda, \qquad\textrm{ for}\quad u\ge 0. 
\]
See  \cite[Chapter 8]{Kyp}
 for further details on the above remarks. 
 Note in particular that it is easy to verify that $\hat{\psi}_{\kappa(\lambda)}'(0+)>0$ and hence the process $K$ under $\mathbb{P}^{\kappa(\lambda)}$ drifts to $-\infty$. According to earlier discussion, this guarantees that also under $\mathbb{P}^{\kappa(\lambda)}$, the process $Z$ goes to $0$ as $t\to\infty$.
\begin{lemma}
Suppose that $\lambda\geq 0$ and that $\kappa(\lambda)>1$, then for all $0<b\leq z$,
\[
\mathbb{E}_z\Big[e^{-\lambda \sigma_{b}}\Big] = \frac{\mathbb{E}^{\kappa(\lambda)} \Big[ (1 +kz I_\infty)^{\kappa(\lambda)}\Big]}{\mathbb{E}^{\kappa(\lambda)} \Big[ (zb^{-1} +kz I_\infty)^{\kappa(\lambda)}\Big]},
\]
where 
\[
I_\infty = \int_0^\infty e^{K_s}{\rm d}s.
\]
\end{lemma}
\begin{proof}
From the absence of negative jumps  we have $Z_{\sigma_b}=b$ on the event $\{\sigma_b<\infty\}$ and in particular
\[
b=\frac{ze^{K_{\sigma_b}}}{1+kz\displaystyle \int_0^{\sigma_b} e^{K_s}\ud s}.
\]
On the other hand, from the Markov property and  the above identity, we have
\[
1+kzI_\infty=1+kz\displaystyle \int_0^{\sigma_b} e^{K_s}\ud s+kze^{K_{\sigma_b}}\displaystyle \int_0^{\infty} e^{K_{\sigma_b+s}-K_{\sigma_b}}\ud s=e^{K_{\sigma_b}}\left(\frac{z}{b}+zkI^\prime_\infty\right),
\]
where $I'_\infty$ is an independent copy of $I_\infty$.

The latter identity and the Escheer transform imply that for $\lambda\ge 0$
\[
\mathbb{E}_z\Big[e^{-\lambda \sigma_b}\Big]=\mathbb{E}^{\kappa(\lambda)} \Big[e^{\kappa(\lambda) K_{\sigma_b}}\Big]=\frac{\mathbb{E}^{\kappa(\lambda)} \left[\left(1+kz I_\infty\right)^{\kappa(\lambda)}\right]}{\mathbb{E}^{\kappa(\lambda)} \left[\left(\displaystyle\frac{z}{b}+zkI_\infty\right)^{\kappa(\lambda)}\right]},
\]
provided  the  quantity
$\mathbb{E}^{\kappa(\lambda)}[(a+kzI_\infty)^{\kappa(\lambda)}]$ is finite, for $a>0$. 

Observe that for $s\geq 1$,
\[
\mathbb{E}^{\kappa(\lambda)}\Big[(a+I_\infty)^s\Big]\le 2^{s -1}\Big(a^s+\mathbb{E}^{\kappa(\lambda)}[I_\infty^s]\Big),
\]
hence it suffices to investigate the finiteness of $\mathbb{E}^{\kappa(\lambda)}[I_\infty^s]$.
According to  Lemma 2.1 in Maulik and Zwart \cite{MaZ} the expectation $\mathbb{E}^{\kappa(\lambda)}[I_\infty^{s}]$ is finite for all $s\ge 0$ such that $-\hat{\psi}_{\kappa(\lambda)}(- s)>0$. Since $\hat{\psi}_{\kappa(\lambda)}(- s)$ is well defined for $\kappa(\lambda)-s\ge 0$, 
then a straightforward computation gives us that $\mathbb{E}^{\kappa(\lambda)}[I_\infty^s]<\infty$ for $s\in [0,\kappa(\lambda)]$. 
\end{proof}
\section{Appendix}
{The following results are useful for  the proofs of Theorem \ref{Existencia Z} and Proposition \ref{Prop1}. Most of the proof of the results that we present here follow the same arguments from similar results that appear in \cite{FuLi} or \cite{LiPu}. We will provide the proof of those results that we believe are more complicated to deduce from \cite{FuLi} or \cite{LiPu}. }

 Given a differentiable function $f$, we write
$$ \Delta_xf(a)=f(x+a)-f(a)\qquad \mbox{and} \qquad D_xf(a)=\Delta_xf(a)-f'(a)x.$$

\noindent Let $(a_n, n\ge 1)$ a sequence of positive real numbers such that $a_0=1$,   $a_n\downarrow0$ and $\int_{a_n}^{a_{n-1}}z\ud z=n$ 
 for each $n\in\N$. Let $x\mapsto \kappa_n(x)$ be a non-negative continuous function supported on $(a_n,a_{n-1})$ such that  $\kappa_n(x)\leq 2(nx)^{-1}$, for every $x>0$, and $\int_{a_n}^{a_{n-1}}\kappa_n(x)\ud x=1.$
	For $\ell\geq 0$, let us define
	$$f_\ell(z)=\int_{0}^{|z|}\ud y\int_0^y \kappa_\ell(x)\ud x, \qquad z\in\R.$$
	Observe that $(f_\ell, \ell\ge 0)$ is a non-decreasing sequence of functions that converges to the {mapping} $x\mapsto |x|$ as $\ell$ {increases}. For all $a,x\in\R$, we have $|f_\ell'(a)|\leq 1$ and  $|f_\ell(a+x)-f_\ell(a)|\leq |x|$. Moreover, by Taylor's expansion, we deduce
		\begin{align*}
		\Big|D_xf_\ell(a)\Big|\leq x^2\int_0^1\kappa_\ell(|a+xu|)(1-u)\ud u\leq \frac{2}{\ell}x^2 \int_0^1\frac{(1-u)}{|a+xu|}\ud u.
		\end{align*}

The proof of the following lemma can be found in \cite{LiPu} (Lemma 3.1).
		
\begin{lemma}\label{lema3.1lipu}
Suppose that $x\mapsto x+h(x,v)$ is non-decreasing for $v\in\mathcal{V}$. Then, for any $ x\neq y\in\R$,
{\begin{align*}
D_{l(x,y,v)}f_m(x-y)\leq \frac{2}{m}\int_0^1\frac{l(x,y,u)^2 (1-u)}{|x-y+ul(x,y,v)|}\ud u\leq \frac{2 l(x,y,v)^2}{m|x-y|},
\end{align*}}
where $l(x,y,v)=h(x,v)-h(y,v)$.
\end{lemma}

The following result shows that pathwise uniqueness holds if the parameters $(b, (\sigma_k)_{k\in K},$
$  (h_i)_{i\in I}, (g_j)_{j\in J})$ are admissible and satisfies conditions a), b) and c). 
\begin{lemma}\label{pathwise}
	The pathwise uniqueness holds for the positive solutions of
	\begin{equation}\label{sdenw1}
	\begin{split}
	Z_t^{(n)}&=Z_0+\int_0^t b(Z_s^{(n)}\wedge n) \ud s+\sum_{k\in K}\int_0^t \sigma_k( Z_s^{(n)}\wedge n)\ud B^{(k)}_s \\
	&\hspace{2.8cm}+\sum_{i\in I}\int_0^t\int_{W_i}\big(g_i(Z_{s-}^{(n)}\wedge n,u_i)\wedge n\big) M_i(\ud s,\ud u_i)\\
	&\hspace{4cm}+\sum_{j\in J}\int_0^t\int_{V_j}\big(h_j(Z_{s-}^{(n)}\wedge n,v_j)\wedge n\big) \widetilde{N}_j(\ud s,\ud v_j),
	\end{split}
	\end{equation}
	for every $n\in\N$.
\end{lemma}

\begin{proof}
We consider $Z_t$ and $Z_t'$ two solutions of (\ref{sdenw1}) and let $Y_t=Z_t-Z_t'$. Therefore, $Y_t$ satisfies the SDE
	\[
	\begin{split}
	Y_t&= Y_0+\int_0^t\Big( b(Z_s\wedge n)-b(Z_s'\wedge n)\Big) \ud s+\sum_{k\in K}\int_0^t\Big( \sigma_k( Z_s\wedge n)-\sigma_k( Z_s'\wedge n)\Big)\ud B^{(k)}_s \\
	&+\sum_{i\in I}\int_0^t\int_{W_i}\widetilde{g}^{(n)}_i(Z_{s-},Z_{s-}',u_i) M_i(\ud s,\ud u_i)+\sum_{j\in J}\int_0^t\int_{V_j}\widetilde{h}^{(n)}_j(Z_{s-},Z_{s-}',v_j)\widetilde{N}_j(\ud s,\ud v_j),
	\end{split}
	\]
	{where 
	\begin{align*}
	\widetilde{g}^{(n)}_i(x,y,u_i)&= g_i(x\wedge n,u_i)\wedge n-g_i(y\wedge n,u_i)\wedge n,\\
	\widetilde{h}^{(n)}_j(x,y,v_j)&=h_j(x\wedge n,v_j)\wedge n-h_j(y\wedge n,v_j)\wedge n.
	\end{align*}}
	By applying It\^o's formula to the functions $f_\ell$, we deduce
	{\begin{equation}\label{trayecto}
	\begin{split}
	f_\ell(Y_t)&=f_\ell(Y_0)+M_t+\int_0^t f_\ell'(Y_s)\Big( b(Z_s\wedge n)-b(Z_s'\wedge n)\Big)\ud s\\
	&\hspace{2.5cm}+\sum_{k\in K}\frac{1}{2}\int_0^tf_\ell''(Y_s) \Big( \sigma_k( Z_s\wedge n)-\sigma_k( Z_s'\wedge n)\Big)^2 \ud s  \\
	&\hspace{3cm}+\sum_{i\in I}\int_0^t\int_{W_i}\Delta_{\widetilde{g}^{(n)}_i(Z_{s-},Z_{s-}',u_i)} f_\ell(Y_{s-})\mu_i(\ud u_i)\ud s\\
	&\hspace{3.5cm}+\sum_{j\in J}\int_0^t\int_{V_j} D_{\widetilde{h}^{(n)}_i(Z_{s-},Z_{s-}',u_i)}f_\ell(Y_{s-})\nu_j(\ud v_j)\ud s,
	\end{split}
	\end{equation}
	where $M_t$ is a martingale. Since $b=b_1-b_2$ as in condition (b), we have 
\[
|f_\ell'(x-y)|| b(x\wedge n)-b(y\wedge n)|\leq |b_1(x\wedge n)-b_1(y\wedge n)|\leq r_n(|x-y|\wedge n),
\]
and
\[
	\sum_{i\in I}\int_{W_i}\Delta_{\widetilde{g}^{(n)}_i(x,y,u_i)}f_\ell(x-y)\mu_i(\ud u_i)\leq \sum_{i\in I}\int_{W_i}|\widetilde{g}^{(n)}_i(x,y,u_i)|\mu_i(\ud u_i)\leq r_n (|x-y|\wedge n).
\]
 Since $x\mapsto x+h_j(x\wedge n,v)\wedge n $ is non-decreasing for  $j\in J$, by Lemma \ref{lema3.1lipu} and condition (c), we deduce 
	\[
	\sum_{j\in J}\int_{V_j}D_{	\widetilde{h}^{(n)}_j(x,y,v_j)}f_\ell(x-y)\nu_j(\ud v_j)\leq \sum_{j\in J}\int_{V_j}\frac{2 \widetilde{h}^{(n)}_j(x,y,v_j)^2}{\ell |x-y|}\nu_j(\ud v_j)\leq \frac{2B_n}{\ell},
	\]
and
	\[
	\sum_{k\in K}f_\ell''(x-y)(\sigma_k(x)-\sigma_k(y))^2\leq \sum_{k\in K}\kappa_\ell(x-y)|\sigma_k(x)^2-\sigma_k(y)^2|\leq \frac{2B_n}{\ell}.
	\]
Hence taking  expectations in both sides of equation (\ref{trayecto}) and putting all the pieces together, we get
	\begin{align*}
	\Exp{f_\ell(Y_t)}\leq&\Exp{f_\ell(Y_0)}+2\int_0^t \Exp{r_n(|Y_s|\wedge n)}\ud s +4\ell^{-1}B_n.
	\end{align*}
	Since $f_\ell(z)$ increases towards $|z|$ as $\ell\rightarrow\infty$, we have
	$$\Exp{|Y_t|}\leq\Exp{|Y_0|}+2\int_0^t \Exp{r_n(|Y_s|\wedge n)}\ud s.$$
	Finally from  Gronwall's inequality, we can deduce that  pathwise uniqueness of solutions holds for (\ref{sdenw1}).}
\end{proof}

Suppose that $\mu_i(U_i\setminus W_i)<\infty$, for all $i\in I$.
Our next result shows that if there is a unique strong solution to (\ref{sdenw1}), we  can replace the spaces $(W_i)_{i\in I}$ by $(U_i)_{i\in I}$ on the SDE and the unique strong solution still exists for the extended SDE. {Its proof follows from similar arguments as those used in Proposition 2.2 in \cite{FuLi}. }
\begin{lemma}\label{subconjuntow} If there is a unique  strong solution to (\ref{sdenw1}) and $\mu_i(U_i\setminus W_i)<\infty$ for all $i\in I$, then there is also a strong solution to 
	\begin{equation*}
	\begin{split}\label{sdenapen1}
	Z_t^{(n)}=Z_0&+\int_0^t b(Z_s^{(n)}\wedge n) \ud s+\sum_{k\in K}\int_0^t \sigma_k( Z_s^{(n)}\wedge n)\ud B^{(k)}_s \\
	&\hspace{1cm}+\sum_{i\in I}\int_0^t\int_{{U_i}}\big(g_i(Z_{s-}^{(n)}\wedge n,u_i)\wedge n\big) M_i(\ud s,\ud u_i)\\
	&\hspace{2cm}+\sum_{j\in J}\int_0^t\int_{V_j}\big(h_j(Z_{s-}^{(n)}\wedge n,v_j)\wedge n\big) \widetilde{N}_j(\ud s,\ud v_j).
	\end{split}
	\end{equation*} 
\end{lemma}

{Next, recall that for each $n,\ell\in \N$, the process $Z_t^{(n,\ell)}$ was defined as the solution to
\begin{equation}
\begin{split}\label{z,n,m1}
Z_t^{(n,\ell)}=Z_0&+\int_0^t b(Z_s^{(n,\ell)}\wedge n) \ud s+\sum_{k\in K}\int_0^t \sigma_k( Z_s^{(n,\ell)}\wedge n)\ud B^{(k)}_s \\
&\hspace{1cm}+\sum_{i\in I}\int_0^t\int_{W_i^\ell}\big(g_i(Z_{s-}^{(n,\ell)}\wedge n,u_i)\wedge n\big) M_i(\ud s,\ud u_i)\\
&\hspace{2cm}+\sum_{j\in J}\int_0^t\int_{V_j^\ell}\big(h_j(Z_{s-}^{(n,\ell)}\wedge n,v_j)\wedge n\big) \widetilde{N}_j(\ud s,\ud v_j).
\end{split}
\end{equation}
Following step by step the proof of Lemma 4.3 in \cite{FuLi}, we will prove that for each $n\in \N$, the sequence $\left\{(Z^{(n,\ell)}_t: t\geq 0): \ell\ge 1\right\}$ is tight in $D(\R_+,\R_{+})$.}

{\begin{lemma}\label{sequence znm}
	For each $n\in \N$ the sequence $\left\{(Z^{(n,\ell)}_t: t\geq 0); \ell \ge 1\right\}$ given by \eqref{z,n,m1} is tight in the Skorokhod space $D(\R_+,\R_{+})$.
\end{lemma}}
\begin{proof}
{Since  $b$ is continuous in $[0,n]$ and from hypothesis (b) and (c), we can show that there exists a constant $\mathcal{K}_n>0$ such that for each $x\leq n$
	\begin{equation}\label{cota k}
	\begin{split}
	b(x)+\underset{k\in K}{\sum} &\sigma_k^2(x)+\underset{i\in I}{\sum}\int_{W_i}|g_i(x,u_i)\wedge n|\mu_i(\ud u_i)\\
	&+\underset{i\in I}{\sum}\int_{W_i}|g_i(x,u_i)\wedge n|^2\mu_i(\ud u_i)+\underset{j\in J}{\sum}\int_{V_j}|h_j(x,v_j)\wedge n|^2\nu_j(\ud v_j)\leq \mathcal{K}_n.
	\end{split}
	\end{equation}}
	Note that if $C_n$ is the maximum of $b$ in $[0,n]$, then $\mathcal{K}_n=C_n+nB_n+(n+1)r_n(n)$. By applying Doob's inequality to the martingale terms in \eqref{z,n,m1}, we have
	\begin{equation*}
	\begin{split}
	\Exp{\underset{s\leq t}{\sup}(Z_{s}^{(n,\ell)})^2}\leq & (2+2m+d+l)^2\left( (Z_0)^2+\Exp{\left(\int_0^tb(Z_s^{(n,\ell)}\wedge n)\ud s\right)^2}\right.\\
	&\hspace{2.8cm}+ \underset{i\in  I}{\sum}\Exp{\left(\int_0^t\int_{W_i} |g(Z_s^{(n,\ell)}\wedge n, u_i)\wedge n|\mu_i(\ud u_i)\right)^2}\\
	& \hspace{-1.5cm}+4\left( \underset{k\in K}{\sum } \int_0^t \sigma_k^2(Z_s^{(n,\ell)}\wedge n)\ud s+\underset{i\in I}{\sum } \int_0^t\int_{W_i} |g_i(Z_s^{(n,\ell)}\wedge n, u_i)\wedge n|^2\mu_i(\ud u_i) \right.\\
	&\hspace{3.8cm}\left.\left.+\underset{j\in J}{\sum } \int_0^t\int_{V_j} |h_j(Z_s^{(n,\ell)}\wedge n, v_j)\wedge n|^2\nu_j(\ud u_j)\right)\right).
	\end{split}
	\end{equation*}
	From inequality \eqref{cota k}, we obtain that the mapping
	$$t\mapsto \sup_{\ell \ge 1} \Exp{\underset{s\leq t}{\sup}\big(Z_{s}^{(n,\ell)}\big)^2}\leq (2+2m+d+l)^2\left( (Z_0)^2+(1+m)\mathcal{K}_n^2t^2+4\mathcal{K}_nt\right),$$
	is  locally bounded. Then for every fixed $t\geq 0$, the sequence $\left\{Z^{(n,\ell)}_t; \ell\ge 1\right\}$ is tight. In a similar way, if $\{\tau_\ell:\ell\geq 1\}$ is a sequence of stopping times bounded above by $T\geq 0$, we have
	$$\Exp{\Big|Z^{(n,\ell)}_{\tau_\ell+t}-Z^{(n,\ell)}_{\tau_\ell}\Big|^2}\leq (2+2m+d+l)^2\left((1+m)\mathcal{K}_n^2t^2+4\mathcal{K}_nt\right).$$
	Consequently, as $t\rightarrow 0$, we have
	$$\sup_{\ell\ge 1}\Exp{\Big|Z^{(n,\ell)}_{\tau_\ell+t}-Z^{(n,\ell)}_{\tau_\ell}\Big|^2}\rightarrow 0.$$
	By Aldous' criterion \cite{al}, for all $n\in \N$, $\left\{(Z^{(n,\ell)}_t: t\geq 0); \ell \ge 1\right\}$ is tight in $D(\R_+,\R_{+})$.	
\end{proof}

For each $n,\ell \in \N$, $x\geq 0$ and $f\in C^2(\R)$ we define
\begin{equation*}\small
\begin{split}\label{generador n}
\mathcal{L}^{(n)}f(x)=&b(x\wedge n)f'(x)+\frac{1}{2}f''(x)\sum_{k\in K}\sigma_k^2(x\wedge n)\\
\small
&+\sum_{i\in I}\int_{W_i} \Delta_{g_i(x\wedge n,u_i)\wedge n}f(x)\mu_i(\ud u_i)+\sum_{j\in J}\int_{V_j} D_{h_j(x\wedge n,v_j)\wedge n}f(x)\nu_j(\ud v_j). 
\end{split}
\end{equation*}
and 
\begin{equation*}\small
\begin{split}\label{generador nm}
\mathcal{L}^{(n,\ell)}f(x)=&f'(x)b(x\wedge n)+\frac{1}{2}f''(x)\sum_{k\in K}\sigma_k^2(x\wedge n)\\
&+\sum_{i\in I}\int_{W_i^\ell}  \Delta_{g_i(x\wedge n,u_i)\wedge n}f(x)\mu_i(\ud u_i)+\sum_{j\in J}\int_{V_j^\ell} D_{h_j(x\wedge n,v_j)\wedge n}f(x)\nu_j(\ud v_j). 
\end{split}
\end{equation*}

Now, we prove the existence of the weak solution of a SDE by considering the corres\-ponding martingale problem.
\begin{lemma}
	\label{mar}
	
	A c\`adl\`ag process $(Z_t^{(n)}: t\geq 0)$ is a weak solution of \eqref{sdenw1} if and only if for every $f\in C^{2}(\R)$,
	\begin{equation}\label{mart n}
	f\Big(Z_t^{(n)}\Big)-f\Big(Z_0^{(n)}\Big)-\int_0^t \mathcal{L}^{(n)}f\Big(Z_s^{(n)}\Big)\ud s
	\end{equation}
	is a locally bounded martingale. Moreover a  c\`adl\`ag process $(Z_t^{(n, \ell)}: t\geq 0)$ is a weak solution of \eqref{z,n,m1} if and only if for every $f\in C^{2}(\R)$,
	\begin{equation}\label{mart n,m}
	f\Big(Z_t^{(n, \ell)}\Big)-f\Big(Z_0^{(n, \ell)}\Big)-\int_0^t \mathcal{L}^{(n, \ell)}f\Big(Z_s^{(n, \ell)}\Big)\ud s
	\end{equation}
	is a locally bounded martingale.
\end{lemma}

\begin{proof}
	We will just prove the first statement, the second one is analogous. If $(Z_t: t\geq 0)$ is a solution of \eqref{sdenw1}, by It\^o's formula we can see that \eqref{mart n} is a locally bounded martingale. Conversely, suppose that \eqref{mart n} is is a locally bounded martingale for every $f\in C^{2}(\R)$. By a stopping time argument, we have
	\begin{equation*}\label{mm}
	Z_t=Z_0+\int_0^tb(Z_t\wedge n)\ud s+\underset{i\in I}{\sum}\int_0^t\int_{W_i}(g_i(Z_{s-}\wedge n,u_i)\wedge n)\mu_i(\ud u_i)\ud s+ M_t
	\end{equation*}
	for a square-integrable martingale $(M_t:t\geq 0)$. Let $N(\ud s,\ud z)$ be the optional random measure on $[0,\infty)\times \R$ defined by
	$$N(\ud s,\ud z)=\underset{s>0}{\sum} \Ind{\Delta Z_s\neq 0}\delta_{(s,\Delta Z_s)}(\ud s,\ud z),$$
	with $\Delta Z_s=Z_s-Z_{s-}$. Denote by 
	$\tilde{N}$ the compensated measure.
	Then
	\begin{equation}\label{repmar}
	Z_t=Z_0+\int_0^tb(Z_t\wedge n)\ud s+\underset{i\in I}{\sum}\int_0^t\int_{W_i}(g_i(Z_{s-}\wedge n,u_i)\wedge n)\mu_i(\ud u_i)\ud s+ M_t^c+M^d_t
	\end{equation}
	where $(M_t^c:t\geq 0)$ is a continuous martingale and 
	$$M_t^d=\int_0^t\int_{\R}z\tilde{N}(\ud s,\ud z)$$
	 is a purely discontinuous martingale. As in the proof of Proposition 4.2 in \cite{FuLi}, 
	we obtain \eqref{sdenw1} on an extension of the probability space by applying Martingale Representation Theorems to \eqref{repmar}. (See \cite[Section II.7]{IW})
\end{proof}	

By Lemma \ref{sequence znm}, $\left\{(Z^{(n,\ell)}_t: t\geq 0); \ell \ge 1\right\}$ is tight in the Skorokhod space $D(\R_+,\R_{+})$. Then, there exists a subsequence $\left\{(Z^{(n,\ell_k)}_t: t\geq 0); k \ge 1\right\}$  that converges to some process $(Z^{(n)}_t: t\geq 0)$ in the Skorokhod sense. By the Skorokhod Representation Theorem, we may assume those processes are defined in the same probability space and  $\left\{(Z^{(n,\ell_k)}_t: t\geq 0); k \ge 1\right\}$ converges to $(Z^{(n)}_t: t\geq 0)$  almost surely in $D(\R_+,\R_+)$.  Let $\mathcal{D}(Z^{(n)})=\left\{t>0: \Prob{Z^{(n)}_{t-}=Z_t^{(n)}}=1\right\}$. Then, $[0,\infty)\setminus \mathcal{D}(Z^{(n)})$ is at most countable and $\lim_{k\rightarrow\infty}Z_t^{(n,\ell_k)}=Z_t^{(n)}$ almost surely for every $t\in \mathcal{D}(Z^{n})$ (see \cite[p. 118,131]{EKmarkov}). 

\begin{lemma} \label{weak limit}
The process $(Z^{(n)}_t: t\geq 0)$ is 
is a weak solution of \eqref{sdenw1}.
\end{lemma}

\begin{proof}
{We first prove that if  $(Y_\ell,\ell\in \N)$ is a a.s. convergent sequence of random variables and $Y$ its is limit, then $ \mathcal{L}^{(n,\ell)}f(Y_\ell)\rightarrow \mathcal{L}^{(n)}f(Y)$ as $\ell\rightarrow \infty$, for every $f\in C^{2}(\R_+)$.}
	Let $M>0$ be a constant such that $|Y|,|Y_\ell|\leq M$, for all $m\in\N$. By conditions (b) and (c), we have that for each $k$, the mapping
	$$x\mapsto \sum_{i\in i}\int_{W_i\setminus W_i^k}(g_i(x\wedge n,u_i)\wedge n)\mu_i(\ud u_i)+
	\sum_{j\in J}\int_{V_j\setminus V_j^k}(h_j(x\wedge n,v_j)\wedge n)^2\nu_j(\ud v_j)$$
	is continuous. By Dini's Theorem, we deduce, as $k\rightarrow\infty$, that
	$$\epsilon_k:=\sup_{|x|\leq M}\int_{W_i\setminus W_i^k}(g_i(x\wedge n,u_i)\wedge n)^2\mu_i(\ud u_i)+
	\sum_{j\in J}\int_{V_j\setminus V_j^k}(h_j(x\wedge n,v_j)\wedge n)\nu_j(\ud v_j)\rightarrow 0.$$
	For $\ell\geq k$ and $j\in J$, by applying the Mean Value Theorem several times we have\footnote{In order to simplify the notation, for a $f\in C(\R)$, we will denote 
		$\|f\|:=\max\{|f(x)|:|x|\leq n+M\}.$
		By continuity, $\|f\|<\infty$. }
		\[
	\begin{split}
	\bigg|\int_{V_j}D_{h(Y\wedge n,v_j)\wedge n}&f(Y)\nu_j(\ud v_j)-\int_{V_j^\ell}D_{h(Y_\ell\wedge n,v_j)\wedge n}f(Y_\ell)\nu_j(\ud v_j)\bigg|\\
	\leq & \|f''\|\epsilon_k+\int_{V_j^k} \Big|f(Y)-f(Y_\ell)\Big|\nu_j(\ud v_j)\\
	& +\|f'\|\int_{V_j^k} \Big|Y+h(Y\wedge n,v_j)\wedge n-Y_\ell-h(Y_\ell\wedge n,v_j)\wedge n\Big|\nu_j(\ud v_j)\\
	& +\|f'\|\int_{V_j^k}\Big |h(Y\wedge n,v_j)\wedge n)-h(Y_\ell\wedge n,v_j)\wedge n\Big|\nu_j(\ud v_j)\\
	&+ \int_{V_j^k}\Big|f'(Y)-f'(Y_\ell)\Big| \Big|h(Y\wedge n,v_j)\wedge n\Big|\nu_j(\ud v_j).
	\end{split}
	\]
	Then, by H\"older inequality we have
	\[
	\begin{split}
	\bigg|\int_{V_j}D_{h(Y\wedge n,v_j)\wedge n}&f(Y)\nu_j(\ud v_j)-\int_{V_j^\ell}D_{h(Y_\ell\wedge n,v_j)\wedge n}f(Y_\ell)\nu_j(\ud v_j)\bigg|\\
	\leq& \|f''\|\epsilon_k+ |f(Y)-f(Y_\ell)|\nu_j(V_j^k) +\|f'\||Y-Y_\ell|\nu_j(V_j^k)\\
	&+2\|f'\|\int_{V_j^k} \Big|h(Y\wedge n,v_j)\wedge n-h(Y_\ell\wedge n,v_j)\wedge n\Big|\nu_j(\ud v_j)\\
	&+ |f'(Y)-f'(Y_\ell)|\left(\nu_j(V_j^k)\int_{V_j}|h(Y\wedge n,v_j)\wedge n|^2\nu_j(\ud v_j) \right)^{1/2}.
	\end{split}
	\]
	By letting $\ell$ and $k$ go to $\infty$, and using hypothesis (c), we have that for all $j\in J$,
	\begin{equation}\label{limiteV_j}
	\underset{\ell\rightarrow \infty}{\lim} \int_{V_j^\ell} D_{h_j(Y_\ell\wedge n,v_j)\wedge n}f(Y_\ell)\nu_j(\ud v_j)=\int_{V_j} D_{h_j(Y\wedge n,v_j)\wedge n}f(Y)\nu_j(\ud v_j).
	\end{equation}
	In a similar way,  by hypothesis (b), we have for  all $i\in I$
	\begin{equation}
	\label{limiteW_j}
	\underset{\ell\rightarrow \infty} {\lim}\int_{W_i^\ell}\Delta_{g_i(Y_\ell\wedge n,u_i)\wedge n} f(Y_\ell)\mu_i(\ud u_i)=\int_{W_i}\Delta_{g_i(Y\wedge n,u_i)\wedge n} f(Y)\mu_i(\ud u_i).  \end{equation}
	Therefore, by \eqref{limiteV_j} and \eqref{limiteW_j} it follows that $ \mathcal{L}^{(n,\ell)}f(Y_\ell)\rightarrow \mathcal{L}^{(n)}f(Y)$ as $\ell\rightarrow\infty$.
	
	\noindent Finally, by the previous step, for all $t\in \mathcal{D}(Z^{(n)})$, we have that
	$ \mathcal{L}^{(n,\ell)}f(Z_t^{(n,\ell_k)})\rightarrow \mathcal{L}^{(n)}f(Z_t)$ as $\ell\rightarrow\infty$.
	{Recall that for all $k\in \N$, $(Z_t^{(n,\ell_k)}: t\geq 0)$ is a weak solution of \eqref{z,n,m1} and  from Lemma \ref{mar}, for each $f\in C^{2}(\R_+)$
	$$	f\Big(Z_t^{(n, \ell_k)}\Big)-f\Big(Z_0^{(n, \ell_k)}\Big)-\int_0^t \mathcal{L}^{(n, \ell_k)}f\Big(Z_s^{(n, \ell_k)}\Big)\ud s$$
	 is a locally bounded martingale. Then, the Dominated Convergence Theorem implies that \eqref{mart n} is a locally bounded martingale.  And, from Lemma \ref{mar} we obtain that $(Z^{(n)}_t:t\geq 0)$ is a weak solution of \eqref{sdenw1}.}
\end{proof}

The following result shows the  a.s. uniqueness of (\ref{bpbackward}) and it is needed for the proof of Proposition 1. 
\begin{lemma}\label{existencev}
Suppose that $\int_{[1,\infty)}x\mu(\ud x)<\infty$ and let $K=(K_t, t\ge 0)$ be a L\'evy process. Then for every $\lambda \ge 0$, $v_t:s\in[0,t]\mapsto v_t(s, \lambda, K)$  is the a.s. unique solution of the backward differential equation,
 \begin{align}\label{bpbackward1}
 \frac{\partial}{\partial s}v_t(s,\lambda, K)=e^{K_s}\psi_0(v_t(s,\lambda,K)e^{-\delta_s}),\qquad v_t(t,\lambda, K)=\lambda,
 \end{align}
 where
 $$\psi_0(\theta)=\psi(\theta)-\theta\psi'(0)=\gamma^2
\theta^2+\int_{(0,\infty)}\big(e^{-\theta x}-1+\theta x\big)\mu(\ud x), \qquad \theta\ge 0.$$
\end{lemma}

 \begin{proof} Our proof  will use a convergence argument for L\'evy processes.  Let $K$ be a L\'evy process with characteristic $(\alpha,\sigma,\pi)$ where $\alpha\in \mathbb{R}$ is the drift term, $\sigma\ge 0$ is the Gaussian part and $\pi$ is the so-called L\'evy measure satisfying $ \int_{\mathbb{R}\setminus \{0\}} (1\land z^2)\pi(\ud z)<\infty.$ 
 From the L\'evy-It\^o decomposition (see for instance \cite{Kyp}), the process $K$ can be decomposed as the sum of three independent processes: $X^{(1)}$ a Brownian motion with drift, $X^{(2)}$ a compound Poisson process and $X^{(3)}$  a square-integrable martingale with an a.s. countable number of jumps on each finite time interval with magnitude less than 1. Let  $B_{\epsilon}=(-1,-\epsilon)\cup (-\epsilon,1)$ and $M$ be  a Poison random measure with characteristic  measure $\ud t\pi(\ud x)$. Observe that the process
$$X_t^{(3,\epsilon)}=\int_{[0,t]}\int_{B_{\epsilon}}xM(\ud s,\ud x)-t\int_{B_{\epsilon}}x\pi(\ud x), \qquad t\geq 0$$
is a martingale. According to Theorem 2.10 in \cite{Kyp}, for a fixed $t\geq 0$, there exists a deterministic subsequence $(\epsilon_n)_{ n\in \N}$ such that $(X_s^{3,\epsilon_n}, 0\le s\leq t)$ converges uniformly to $(X_s^3, 0\le s\le t)$ with probability one.
We now define $$K^{(n)}_s=X_s^{(1)}+X_s^{(2)}+X_s^{(3,\epsilon_n)}, \qquad s\leq t.$$

In the sequel, we  work on the space $\widetilde\Omega$ such that $K^{(n)}$ converges uniformly to $K$ on $[0,t].$  Note that $\psi_0$ is locally Lipschitz and $K^{(n)}$ is a piecewise continuous function with a finite number of discontinuities. Hence, from the Cauchy-Lipschitz Theorem, we can define a unique solution $v^n_t(\cdot,\lambda, K^{(n)})$ of the backward differential equation:
\begin{align*}
 \frac{\partial}{\partial s}v_t^n(s,\lambda, K^{(n)})=e^{K^{(n)}_s}\psi_0(v_t^n(s,\lambda, K^{(n)})e^{-K^{(n)}_s}),\qquad v_t^n(t,\lambda, K^{(n)})=\lambda.
 \end{align*}
 In order to prove our result, we  show that the sequence $(v^n(s):=v^n_t(s,\lambda, K^{(n)}),\ s\leq t)_{n\in\N}$ converges to a unique solution of (\ref{bpbackward1}) on $\widetilde{\Omega}$.  With this purpose in mind, we define
 \begin{align}\label{defs}
 S=\underset{s\in[0,t],\ n\in\N}{\sup}\left\{e^{K^{(n)}_s},e^{-K^{(n)}_s},e^{K_s},e^{-K_s}\right\},
 \end{align}
 which turns out to be finite from the uniform convergence of $K^{(n)}$ to $K$. Since $\psi_0\geq 0$, we  have that 
 $v^n$ is  increasing  and 
 $v^n(s)\leq \lambda $ for $s\leq t$ and $n\in \N$.
 On the other hand,  since $\psi_0$ is a convex and increasing, we deduce that  for any $0\leq \zeta\leq \eta\leq \lambda S$, the following inequality holds
 \begin{align}\label{cotapsi}
 0\leq\frac{\psi_0(\eta)-\psi_0(\zeta)}{\eta-\zeta}\leq \psi_0'(\eta)\leq \psi_0'(\lambda S)=:C.
 \end{align}
 For simplicity, we denote for all $v\ge 0$, 
 $$\psi^n(s,v)=e^{K^{(n)}_s}\psi_0(ve^{-K^{(n)}_s}) \quad \mbox{and} \quad \psi^{\infty}(s,v)=e^{K_s}\psi_0(ve^{-K_s}).$$
 We then observe that for any $0\leq s\leq t$ and $n,m\in\N$,
 \begin{align*}
 |v^n(s)-v^m(s)|
\leq&\int_s^t(R^n(u)+R^m(u))\ud u+\int_s^t|\psi^{\infty}(u,v^n(u))-\psi^{\infty}(u,v^m(u))|\ud u,
\end{align*}
where for any $u\in[0,t]$,
\begin{align*}
R^n(u):=&|\psi^{n}(u,v^n(u))-\psi^{\infty}(u,v^n(u))|\\
\leq& e^{K^{(n)}_u}|\psi_0(v^n(u)e^{-K^{(n)}_u}) -\psi_0(v^n(u)e^{-K_u})|+\psi_0(v^n(u)e^{-K_u})|e^{K^{(n)}_u}-e^{K_u}|.
\end{align*}
Next, using (\ref{defs}) and (\ref{cotapsi}), we deduce
\begin{align*}
R^n(u)&\leq SC\lambda |e^{-K^{(n)}_u}-e^{-K_u}|+\psi_0(S\lambda)|e^{K^{(n)}_u}-e^{K_u}|\\
&\leq \left(SC\lambda+S\psi_0(S\lambda)\right)\underset{u\in[0,t]}{\sup}\left\{|e^{K^{(n)}_u}-e^{K_u}|,|e^{-K^{(n)}_u}-e^{-K_u}|\right\}=:s_n.
\end{align*}
From similar arguments, we obtain
$$|\psi^{\infty}(u,v^n(u))-\psi^{\infty}(u,v^m(u))|\leq C|v^n(u)-v^m(u)|.$$
Therefore,
$$|v^n(s)-v^m(s)|\leq R_{n,m}(s)+C\int_s^t |v^n(u)-v^m(u)|\ud u,$$
where 
$$R_{n,m}(s)=\int_s^t(R^n(u)+R^m(u))\ud u.$$
Gronwall's lemma yields that for all $0\leq s\leq t,$
$$|v^n(s)-v^m(s)|\leq R_{n,m}(s)+C\int_s^t R_{n,m}(u)e^{C(u-s)}\ud u.$$
Now, recalling that $R^n(u)\leq s_n$ and $R_{n,m}(u)\leq(s_n+s_m)t$, we get that for every $N\in\N,$
$$\underset{n,m\geq N, s\in[0,t]}{\sup}|v^n(s)-v^m(s)|\leq te^t\underset{n,m\geq N}{\sup} (s_n+s_m).$$
Moreover since  $s_n\rightarrow0$, we deduce that  $(v^n(s),s\leq t)_{n\in\N}$ is a Cauchy sequence under the uniform norm on $\widetilde\Omega$. In other words, for any $\omega\in\widetilde{\Omega}$ there exists a continuous function $v^*$ on $[0,t]$ such that $v^n\rightarrow v^*$ as $n$ goes to $\infty$. 
We define the function  $v:\Omega\times[0,t]\rightarrow[0,\infty]$ as follows
\[ 
v(s)=
\left\{ \begin{array}{ll}
v^*(s) & \mbox{if } \omega\in\widetilde{\Omega}, \\
0 &\mbox{elsewhere.} 
\end{array}\right.
\] 
Let $s\in[0,t]$ and $n\in\N$, then 
\begin{align*}
\left|v(s)-\int_s^t\psi^{\infty}(s,v(s))\ud s-\lambda\right|\leq &|v(s)-v^n(s)|+\int_s^t|\psi^n(s,v(s))-\psi^n(s,v^n(s))|\ud s\\
&\hspace{3cm}+\int_s^t|\psi^{\infty}(s,v(s))-\psi^n(s,v(s))|\ud s\\
\leq& (1+Ct)\underset{s\in[0,t]}{\sup}\left\{|v(s)-v^n(s)|\right\}+ts_n.
\end{align*}
By letting $n\rightarrow \infty$, we obtain our claim. The uniqueness of the solution of (\ref{bpbackward1}) follows from Gronwall's lemma. 
  \end{proof}
  
\noindent \textbf{Acknowledgements}\\

\noindent Both authors  acknowledge support from  the Royal Society and SP also acknowledge  support from  CONACyT-MEXICO Grant 351643.

\end{document}